\documentclass[12pt]{article}        

 \usepackage[T1]{fontenc}              
 \usepackage[latin1]{inputenc}         
 \usepackage[width=16cm, height=22cm]{geometry}  

 \usepackage{amsmath,amssymb,amsthm}   
\usepackage{tcolorbox}        
\usepackage{hyperref} 
\tcbuselibrary{theorems,skins,breakable}             
 \usepackage{xcolor}

\usepackage{authblk}
\usepackage[normalem]{ulem}
\usepackage{enumerate}                   
\usepackage{url}                     
\usepackage{mathrsfs}               
\usepackage{tikz}                    
 \usepackage{tikz-cd}   
\usepackage[width=0.9\textwidth,font=small,labelfont=bf,labelsep=period]{caption}   
               
\usepackage{framed}
\usepackage{mathtools}

\usepackage{verbatim}
\usepackage{slashed}
\usepackage{stmaryrd}
\usepackage{mathabx}
\usepackage{needspace}
\usepackage{sidecap}

\usetikzlibrary{arrows,decorations.pathmorphing,backgrounds,positioning,fit,bending,calc}

\newcounter{comments}
\newenvironment{displaycomment}{\begin{list}{}{\rightmargin=1cm\leftmargin=1cm}\item\sf\begin{small}}{\end{small}\end{list}}

 \usepackage[numbers,square]{natbib}
\newtheoremstyle{remark}
  {6pt}      
  {10pt}      
  {\normalfont} 
  {0pt}      
  {\bfseries}
  {.}        
  {5pt plus 1pt minus 1pt} 
  {}         

\newtheoremstyle{theorem}
  {0pt}      
  {-10pt}      
  {\itshape} 
  {0pt}      
  {\bfseries}
  {.}        
  {5pt plus 1pt minus 1pt} 
  {}         
  
\newcommand{\theoremhrule}{\noindent\rule{\linewidth}{0.5pt}}

\theoremstyle{remark}    
\newtheorem{remark}{Remark}[section]
\newtheorem{definition}[remark]{Definition}
\newtheorem{notation}[remark]{Notation}
 
\newtheorem{example}[remark]{Example}

\theoremstyle{theorem}  
\newtheorem{theoremInBox}[remark]{Theorem}
\newtheorem{observationInBox}[remark]{Observation}
\newtheorem{lemmaInBox}[remark]{Lemma}  
\newtheorem{propositionInBox}[remark]{Proposition}   
\newtheorem{corollaryInBox}[remark]{Corollary}   
\newtheorem{conjectureInBox}[remark]{Conjecture}

\newenvironment{theorem}[1][]{%
\needspace{5\baselineskip}
    \par\bigskip\theoremhrule\par\smallskip
    \begin{theoremInBox}[#1]}%
    {\end{theoremInBox}\par\smallskip\theoremhrule\par
\needspace{5\baselineskip}\bigskip}

    {\end{observationInBox}\par\smallskip\theoremhrule\par
\needspace{5\baselineskip}\bigskip}

\newenvironment{lemma}[1][]{%
\needspace{5\baselineskip}
    \par\bigskip\theoremhrule\par\smallskip
    \begin{lemmaInBox}[#1]}%
    {\end{lemmaInBox}\par\theoremhrule\par
\needspace{5\baselineskip}\bigskip}

\newenvironment{proposition}[1][]{%
    \par\bigskip\theoremhrule\par\smallskip
    \begin{propositionInBox}[#1]}%
    {\end{propositionInBox}\par\smallskip\theoremhrule\par
\needspace{5\baselineskip}\bigskip}

\newenvironment{corollary}[1][]{%
    \par\bigskip\theoremhrule\par\smallskip
    \begin{corollaryInBox}[#1]}%
    {\end{corollaryInBox}\par\smallskip\theoremhrule\par
\needspace{5\baselineskip}\bigskip}

    {\end{conjectureInBox}\par\smallskip\theoremhrule\par
\needspace{5\baselineskip}\bigskip}

\tcolorboxenvironment{proof}{
blanker,breakable,left=3.5mm,
before skip=14pt,after skip=14pt, 
parbox=false, 
borderline west={1mm}{0pt}{black}}

\numberwithin{equation}{subsection}

\newcommand{\cB}{\mathcal{B}}

\newcommand{\cH}{\mathcal{H}}

\newcommand{\cW}{\mathcal{W}}
\newcommand{\cX}{\mathcal{X}}
\newcommand{\cY}{\mathcal{Y}}
\newcommand{\cZ}{\mathcal{Z}}

\newcommand{\sA}{\mathscr{A}}
\newcommand{\sB}{\mathscr{B}}

\newcommand{\sJ}{\mathscr{J}}
\newcommand{\sK}{\mathscr{K}}
\newcommand{\sL}{\mathscr{L}}

\newcommand{\R}{\mathbb{R}}

\newcommand{\N}{\mathbb{N}}
\newcommand{\Z}{\mathbb{Z}}

\newcommand{\C}{\mathbb{C}}

\DeclareMathOperator{\supp}{\operatorname{supp}}

\DeclareMathOperator{\ind}{\operatorname{ind}}

\DeclareMathOperator{\ch}{\operatorname{ch}}

\DeclareMathOperator{\Str}{\operatorname{Str}}
\DeclareMathOperator{\sgn}{\operatorname{sgn}}
\DeclareMathOperator*{\colim}{\mathrm{colim}}
\newcommand{\tr}{\mathrm{tr}}
\newcommand{\Tr}{\mathrm{Tr}}

\newcommand{\GL}{\mathrm{GL}}

\title{Large-scale quantization of trace I:
\\
Finite propagation operators}
\author[1]{Matthias Ludewig}
\author[2]{Guo Chuan Thiang}
\affil[1]{University of Greifswald, Germany}
\affil[2]{Beijing International Center for Mathematical Research, Peking University, China}

\begin{document}

\maketitle

\begin{abstract}
Inspired by parallel developments in coarse geometry in mathematics and exact macroscopic quantization in physics, we present a family of general trace formulas which are universally quantized and depend only on large-scale geometric features of the input data. 
They generalize, to arbitrary dimensions, formulas found by Roe in his partitioned manifold index theorem, as well as the Kubo and Kitaev formulas for 2D Hall conductance used in physics. 
\end{abstract}

\section{Introduction}

J.~Roe's partitioned manifold index theorem \cite{Roe-partition} considers an odd-dimensional noncompact Riemannian manifold $M=M_0\cup M_1$ with compact partitioning hypersurface $N=M_0\cap M_1=\partial M_0$, and proves that a certain trace formula for a Dirac-type operator $D$ on $M$ is equal to the (integer-valued) Fredholm index of the $N$-restricted version of $D$. 
The input for the trace formula is a unitary operator representing the \emph{coarse}, or \emph{large-scale}, index of $D$. 

Independently, in physics, the Hall conductance of a \emph{macroscopic} two-dimensional sample $M$ subject to a uniform perpendicular magnetic field was experimentally found to be exactly quantized to integer multiples of some universal physical constants. 
A further astonishing feature of this quantization is its insensitivity to small-scale perturbations. 
The \emph{Kubo trace formula} 
\begin{equation}
\label{KuboIntron=2}
2 \pi i \cdot \Tr\big(P \big[[X, P], [Y, P]\big]\big) \in \Z
\end{equation}
is usually used for this Hall conductance, and it takes as input a spectral projection of the magnetic Schr\"{o}dinger operator on $\R^2$, and $X, Y \subseteq \R^2$ are the right, respectively upper half plane (identified with the operator that multiplies by their indicator functions).
At first glance, one would probably not expect that the Kubo formula \eqref{KuboIntron=2} takes on only integer values. 
Indeed, much effort has been devoted to showing this surprising fact by relating it to Fredholm indices, see, e.g., \cite{ASS2} and \cite{BES} for the noncommutative geometry perspective. 

In \cite[\S C.1]{Kitaev}, A.~Kitaev proposed an interesting alternative trace formula for the Hall conductance. 
He considered a partition of $M=\Z^2$ into three sectors $A$, $B$ and $C$ (see Fig.~\ref{fig:SmallLargeTriangles}), and stated the formula
\begin{equation}
\label{KitaevIntron=2}
12\pi i\cdot \Tr\big(APBPCP - APCPBP\big)  \in \Z,
\end{equation}
arguing that it was invariant under certain modifications of the partition, and that it equals the Kubo formula and should therefore be integral as well.
Again, in formula \eqref{KitaevIntron=2}, $A$ denotes the operator that multiplies by the indicator function of the subset $A$ and similarly for $B$ and $C$.

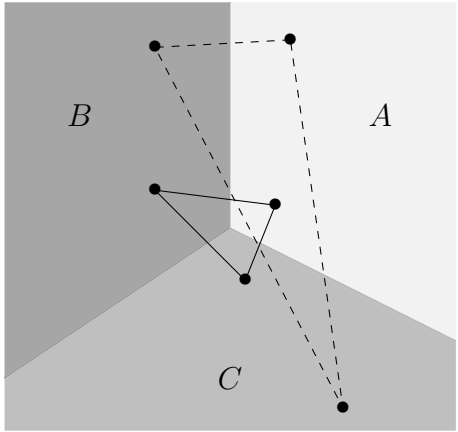
\begin{SCfigure}
\begin{tikzpicture}
\draw (-3,-2)--(0,0);
\draw (0,0)--(3,-1.5);
\draw (0,0)--(0,3);
\fill[color=gray!70] (-3,-2) -- (-3,3) -- (0, 3) -- (0, 0) ;
\fill[color=lightgray] (-3,-2) -- (-3,-2.7) -- (3, -2.7) -- (3,-1.5) -- (0, 0) ;
\fill[color=gray!10] (0, 0) -- (0, 3) -- (3, 3) -- (3, -1.5);
\node at (2,1.5) {$A$};
\node at (-2,1.5) {$B$};
\node at (0,-2) {$C$};
\draw (0.6,0.3) -- (-1,0.5) -- (0.2,-0.7) -- cycle;
\draw[dashed] (0.8,2.5) -- (-1,2.4) -- (1.5,-2.4) -- cycle;
\node at (0.6,0.3) {$\bullet$};
\node at (-1,0.5) {$\bullet$};
\node at (0.2,-0.7) {$\bullet$};
\node at (0.8,2.5) {$\bullet$};
\node at (-1,2.4) {$\bullet$};
\node at (1.5,-2.4) {$\bullet$};
\end{tikzpicture}
\caption{The black dots represent a discrete metric subspace $M$ of the plane, and $M=A\sqcup B\sqcup C$ is a coarsely transverse 2-partition.
Each oriented triangle with one vertex in each of $A, B, C$ represents a signed contribution to the Kitaev pairing \eqref{KitaevIntron=2}. 
 Due to the finite propagation of $P$, large triangles (dashed) with edges longer than the propagation bound do not contribute.
}  \label{fig:SmallLargeTriangles}
\end{SCfigure}

\paragraph{Main results.}

In this paper, we use methods from coarse geometry to analyze the Kubo formula and to make Kitaev's formula precise in a much more general setting. 
We work on arbitrary proper metric spaces $M$, allow collections $X_1, \dots, X_n$ of multiple half-spaces in $M$ for arbitrary $n$, and define generalizations of the above Kubo pairing formula under a certain \emph{coarse transversality} assumption on the half-space collection (see Definition~\ref{dfn:partition} below).
Our main result is then the \emph{quantization} of the \emph{generalized Kubo pairing}
\begin{equation}
\label{KuboPairingIntro}
\frac{(2 \pi i)^{n/2}}{(n/2)!} \cdot \Tr\left(\sum_{\sigma \in S_n} \sgn(\sigma) P X_{\sigma_1} P \cdots P X_{\sigma_n} P\right) \in \Z,
\end{equation}
where $P$ is an arbitrary finite propagation idempotent satisfying a local trace-class condition.
Similarly, for a coarsely transverse collection $A_0, \dots, A_n$ which partitions the space $M$, we have the \emph{quantized generalized Kitaev pairing}
\begin{equation}
\label{KitaevPairingIntro}
\frac{(2 \pi i)^{n/2}n!}{(n/2)!} \cdot \Tr\left(\sum_{\sigma \in S_{n+1}} \sgn(\sigma) A_{\sigma_0} P A_{\sigma_1} P \cdots P A_{\sigma_n} P\right) \in \Z.
\end{equation}
Notice the extra $n!$ factor in \eqref{KitaevPairingIntro} compared to \eqref{KuboPairingIntro}. 
After accounting for cyclicity of the trace,
this corresponds to the extra factor of 2 in Kitaev's formula \eqref{KitaevIntron=2} compared to the $n=2$ Kubo formula \eqref{KuboIntron=2}. 
This $n!$ factor is due to an equipartition principle relating the Kubo and Kitaev pairings (Corollary \ref{CorEquipartition}).
Our technique is, moreover, rather insensitive to the nature of the underlying space, allowing to prove this result for operators over arbitrary proper metric spaces $M$, not just over $M=\R^n$ or $M=\Z^n$.
Whether the set of values of \eqref{KuboPairingIntro}, respectively  \eqref{KitaevPairingIntro}, is all of $\Z$ generally depends on both $M$ and the partition, respectively collection of half-spaces. 
Our result is optimal in the sense that in both cases, any integer can occur as a value of \eqref{KuboPairingIntro} and  \eqref{KitaevPairingIntro} in the case of the standard partition/half-space collection in even-dimensional Euclidean space.

The trace pairings \eqref{KuboPairingIntro} and \eqref{KitaevPairingIntro} can only be non-trivial if $n$ is even.
If $n$ is odd, there are similar pairings, but involving unitary operators $U$ instead of idempotents $P$.
Our formulas unify Roe's formula $(n=1)$, the Kubo formula \eqref{KuboIntron=2} and Kitaev's formula \eqref{KitaevIntron=2} (both $n=2$) into a general framework applicable for all $n$ and all $M$, and are not limited to indices of Dirac-type operators. 

\paragraph{Method of proof.}

Our method of proof will show that the values of these traces depend only on the large-scale aspects of the input data.
More precisely, any coarsely transverse collection of half-spaces and, similarly, any coarsely transverse partition determines a \emph{coarse cohomology class}.
On the other hand, the idempotent $P$ determines a class in the $K$-theory group of an operator algebra of finite propagation operators, which through a suitable Chern character map provides a coarse homology class.
An immediate consequence of this interpretation, summarized in Lemma~\ref{LemmaPartitionInTermsOfChernCharacters}, is that the formulas depend only on the $K$-theory class of $P$ and the coarse cohomology class of the collection of half-spaces, respectively the partition.
However, more profoundly, this  homological interpretation of \eqref{KuboPairingIntro} and \eqref{KitaevPairingIntro}  makes them amenable to the powerful techniques of homological algebra and homotopy theory, which we use to prove the above quantization results.
While an abstract functional analytic proof is possible in the case of $n=2$, it seems hard to conceive how to prove the case of general $n$ in a purely analytic way.

The coarse geometric interpretation of \eqref{KuboPairingIntro} and \eqref{KitaevPairingIntro} is cleanest when taking $P$ to have finite propagation, as carried out in this paper.
However, to handle the spectral projections that arise in physics applications, we should allow for idempotents with the weaker condition of approximate finite propagation. 
Extra conditions would be needed for the local trace norms of such idempotents, as well as the partitions, to ensure that the trace formulas are well-defined. 
The precise modifications needed to handle this situation will be detailed in a separate paper \cite{LudewigThiangPoly}. 
We emphasize that neither of the two situations is contained in the other, and it is of independent mathematical interest to study how traces can get quantized in the strictly finite propagation case.

\paragraph{History.}

The consideration of total anti-symmetrizations of operators as in  \eqref{KitaevPairingIntro} has a long history, going back at least to the 1970s (see, e.g., the work \cite{HeltonHowe1973} of Helton and Howe).
Similarly, the formula \eqref{KuboPairingIntro} may be rewritten in a form closely resembling the formula of Connes for the Chern character of $p$-summable Fredholm modules \cite{Connes}.
However,  all references that we are aware of consider a situation with  a chain of ideals that is \emph{linearly ordered} (such as the collection of Schatten ideals within the algebra of compact operators). 
In contrast, the iterated products considered in our paper are contained in ideals whose inclusion order is governed by the structure of the underlying geometric space $M$ and is \emph{not} linear.
We are not aware of any previous results in such a scenario.
The (non-)relation to the theory of Connes is discussed in more detail in Section \ref{ConnesTheory}.

In recent years, coarse geometry, coarse (co)homology, and Roe algebras have become important mathematical tools in physics; see \cite{EwertMeyer, KS2, Kubota, LTHyperbolic, LTWannier} for some direct applications, and \cite{Kitaev, GNVW} where such ideas were implicitly used. 
The quantized trace formulas presented in this paper add another contribution to this growing list.

\paragraph{Structure of the paper.}

While this paper has ended up being rather lengthy, the conceptual core of the argument is rather short.
Section~\ref{KuboKitaevPairings} introduces the (new) notion of coarse transversality and contains the precise definition of the Kitaev and Kubo pairings.
Sections \ref{SectionCoarseCohomology} \& \ref{Section4} elaborate on the abstract toolbox that we employ to prove our results: Coarse cohomology, algebraic $K$-theory and the Chern character.
Readers familiar with these notions (or those courageous enough to delve right in) may want to jump directly to Section~\ref{ProofQuantization}, which contains the actual arguments, and consult earlier sections for details of the toolbox when the need arises. 

In Sections~\ref{SectionHalfSpaceClasses} \&  \ref{SectionPartitionClasses}, we explain how partitions and collections of half-spaces give rise to coarse cohomology classes.
In Section~\ref{SectionIntegralityKubo}, we use a  dimension reduction argument to prove the integrality result for the Kubo pairing.
The integrality result for the Kitaev pairing is derived from the Kubo result by combining our ``equipartition principle'' from Section~\ref{SectionPartitionClasses} with a ``classifying map'' argument from Section~\ref{SectionIntegralityKitaev}.
In the final Section~\ref{SectionDualIndexTheorem}, we give an index theoretic interpretation of the Kubo pairing and show its non-triviality in the case of $M=\R^n$.

\paragraph{Acknowledgements.}

The authors would like to thank Ulrich Bunke, Alexander Engel and Ralf Meyer for helpful conversations.
M.\ L.\ thanks SFB 1085 ``Higher invariants'' for support.

\tableofcontents

\section{Kubo and Kitaev trace pairings}
\label{KuboKitaevPairings}

In this paper, we work with \emph{proper} metric spaces $M$, meaning that bounded sets are precompact.
The distance function will be denoted by $d$.
Throughout, we will use the following notations.

\begin{notation}[Indicator functions]
\label{IndicatorFunction}
For a subset $Y\subseteq M$, we use the same letter $Y$ to denote its indicator function $Y : M \to \{0, 1\}$, taking the value $1$ on $Y$ and $0$ on its complement.
So the complement $Y^c=M\setminus Y$ corresponds to the function $Y^c = \mathbf{1} - Y$, where $\mathbf{1}$ denotes the constant function with value $1$ on all of $M$.
\end{notation}

\begin{notation}[Thickenings]
The \emph{$R$-thickening} of a subset $Y\subset M$ is the (closed) set of points lying within distance $R$ from $Y$ and will be denoted by
\[
Y_R:=\{x\in M\,:\,d(x,Y)\leq R\}.
\]
\end{notation}

We note that if $X, Y \subseteq M$, then
\begin{equation}
\label{UnionIntersectionInclusions}
(X \cup Y)_R = X_R\cup Y_R,\qquad \text{and} \qquad (X \cap Y)_R \subseteq X_R\cap Y_R.
\end{equation}

\subsection{Coarsely transverse partitions and half-spaces}

Let $M$ be a proper metric space.

\begin{definition}[Big families]
A \emph{big family} $\mathcal{Y}$ in $M$ is a collection of subsets of $M$ that is closed under taking subsets, finite unions, and $R$-thickenings for any $R>0$.
\end{definition}

\begin{example}
The smallest non-empty big family in $M$ is the collection $\cB$ of all bounded subsets of $M$.
\end{example}

\begin{example}
Any subset $Y\subseteq M$ generates a big family consisting of those sets that are contained in some thickening of $Y$. This big family is denoted
\begin{equation*}
\{Y\} := \{Z \subseteq M \mid \exists R > 0 : Z \subseteq Y_R\}.
\end{equation*}
For example, $\{M\}$ is the collection of all subsets of $M$, and this is the largest big family in $M$.
\end{example}

Given two big families $\cX, \cY$ in $M$, the elementwise intersection and union
\begin{align*}
\mathcal{X} \Cap \mathcal{Y} &:= \big\{ X \cap Y \mid X \in \mathcal{X}, Y \in \mathcal{Y}\big\},
\\
\mathcal{X} \Cup \mathcal{Y} &:= \big\{ X \cup Y \mid X \in \mathcal{X}, Y \in \mathcal{Y}\big\},
\end{align*}
are again big families, as follows directly from \eqref{UnionIntersectionInclusions}.
In particular,  we will often consider the elementwise intersection
\[
\partial X:=\{X\}\Cap\{X^c\},
\]
for subsets $X \subseteq M$, which we call the \emph{coarse boundary} of $X$.
This is not to be confused with the topological boundary of $X$ inside $M$, which will not play a role in this paper.

\begin{definition}[Coarse transversality]
A collection $Y_1,\ldots,Y_n\subseteq M$ of subsets of $M$ is \emph{coarsely transverse} if 
\[
\{Y_1\}\Cap\dots\Cap\{Y_n\}\quad \text{is the big family $\cB$ of bounded subsets.}
\]
Equivalently, the set $(Y_1)_R\cap\dots\cap (Y_n)_R$ is bounded for each $R \geq 0$.
\end{definition}

\begin{definition}[Coarsely transverse partitions] 
\label{dfn:partition}
A \emph{coarsely transverse $n$-partition} of $M$ is an ordered collection $A_0, \dots, A_n \subseteq M$ of pairwise disjoint Borel subsets whose union is $M$ and which are coarsely transverse.
\end{definition}

Coarsely transverse 1-partitions of a manifold were studied by J.\ Roe in \cite{Roe-partition} in the context of his partition index theorem.
Some 2-partitions are illustrated in Fig.\ \ref{fig:coarse.transverse} below.

\begin{definition}[Coarsely transverse half-spaces]
\label{dfn:transverse.half.spaces}
A \emph{coarsely transverse collection of half-spaces} is an ordered collection of Borel subsets $X_1,\ldots,X_n \subseteq M$ such that the subsets $X_1, X_1^c, \ldots, X_n, X_n^c$ are coarsely transverse. 
Equivalently, 
\[
\partial X_1 \Cap \cdots \Cap \partial X_n \quad \text{is the big family $\cB$ of bounded sets.}
\]
More generally, a coarsely transverse collection of \emph{switch functions} is a collection of Borel functions $\chi_1, \dots, \chi_n :M\to\C$ such that the sets $\supp(\chi_i)$, $\supp(\mathbf{1} - \chi_i)$, $i=1, \dots, n$, are coarsely transverse.
\end{definition}

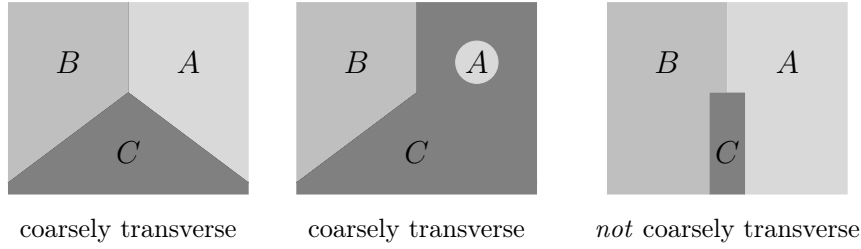
\begin{figure}
\begin{center}
\begin{tikzpicture}[scale=0.8]
\draw (0,0) -- (0,1.5);
\draw (0,0) -- (2,-1.5);
\draw (0,0) -- (-2,-1.5);
\fill[color=gray!30] (0,0) -- (2,-1.5) -- (2,1.5) -- (0,1.5);
\fill[color=lightgray] (0,0) -- (-2,-1.5) -- (-2,1.5) -- (0,1.5);
\fill[color=gray] (0,0) -- (-2,-1.5) -- (-2,-1.7) -- (2,-1.7) -- (2,-1.5);
\node at (1,0.5) {$A$};
\node at (-1,0.5) {$B$};
\node at (0,-1) {$C$};
\node at (0,-2.3) {\footnotesize coarsely transverse};
\end{tikzpicture}
~\hspace{0.5em}
\begin{tikzpicture}[scale=0.8]
\draw (0,0) -- (0,1.5);
\draw (0,0) -- (2,-1.5);
\draw (0,0) -- (-2,-1.5);
\fill[color=lightgray] (0,0) -- (-2,-1.5) -- (-2,1.5) -- (0,1.5);
\fill[color=gray] (0,0) -- (-2,-1.5) -- (-2,-1.7) -- (2,-1.7) -- (2,1.5) -- (0,1.5) -- (0,0);
\filldraw[color=gray!30] (1,0.5) circle (10pt);
\node at (1,0.5) {$A$};
\node at (-1,0.5) {$B$};
\node at (0,-1) {$C$};
\node at (0,-2.3) {\footnotesize coarsely transverse};
\end{tikzpicture}
~\hspace{0.5em}
\begin{tikzpicture}[scale=0.8]
\fill[color=gray] (0,0) -- (0.3,0) -- (0.3,-1.7) -- (-0.3,-1.7) -- (-0.3,0) -- (0,0);
\fill[color=lightgray] (0,0) -- (-0.3,0) -- (-0.3,-1.7) -- (-2,-1.7) -- (-2,1.5) -- (0,1.5) -- (0,0);
\fill[color=gray!30] (0,0) -- (0.3,0) -- (0.3,-1.7) -- (2,-1.7) -- (2,1.5) -- (0,1.5) -- (0,0);
\node at (1,0.5) {$A$};
\node at (-1,0.5) {$B$};
\node at (0,-1) {$C$};
\node at (0,-2.3) {\footnotesize \emph{not} coarsely transverse};
\end{tikzpicture}
\end{center}
\caption{The first diagram shows a coarsely transverse 2-partition $A,B,C$ of the plane. 
For any $R>0$, the thickenings $A_R, B_R, C_R$ will intersect within a finite-radius ball centred at the triple intersection point. The second diagram also shows a coarsely transverse 2-partition, but $B$ and $C$ are redundant --- the bounded set $A$ alone guarantees the coarse transversality condition. The third diagram shows a 2-partition which is not coarsely transverse --- for large enough $R$, the thickenings $A_R,B_R,C_R$ have unbounded triple intersection.}\label{fig:coarse.transverse}
\end{figure}

\begin{example}
\label{ExampleEuclideanHalfSpaces}
The \emph{standard} half-spaces on $M=\R^n$ are
\begin{equation*}
	X_i^{\mathrm{std}}=\{(x_1,\ldots,x_n):x_i\geq 0\},\qquad i=1,\ldots,n,
\end{equation*}
and it is obvious that they are coarsely transverse half-spaces in $\R^n$.
To obtain a coarsely transverse partition of $\R^n$, we may set (see Fig.~\ref{fig:standard.partition} below for the $n=2$ case)
\begin{equation}
\begin{aligned}
A_0&=X_1^{\mathrm{std}} \cap \cdots \cap X_n^{\mathrm{std}},\\
A_i&=(X_i^{\mathrm{std}})^c \cap  X_{i+1}^{\mathrm{std}} \cap \cdots  \cap X_n^{\mathrm{std}},\qquad i=1,\ldots,n.
\end{aligned}\label{EuclideanStandardPartition}
\end{equation}
Explicitly,
\[
\begin{aligned}
A_0 &= \left\{(y_1,\ldots,y_n)\in\R^n  \mid  y_{1},\ldots,y_n\geq 0\right\}
\\
A_i &= \left\{(y_1,\ldots,y_n)\in\R^n\mid y_i < 0,\;y_{i+1},\ldots,y_n\geq 0\right\},\qquad i=1,\ldots,n.
\end{aligned}
\]
\end{example}

\subsection{Finite propagation operators}\label{SecFinitePropOperators}

Let $\mathcal{H}$ be an \emph{$M$-module}, by which we mean a separable Hilbert space with a $*$-representation of the algebra $C_0(M)$ of continuous functions vanishing at infinity on $M$.
For example, if $M$ is a smooth Riemannian manifold, then $\cH=L^2(M)$ is a $C_0(M)$-module, while if $M$ is a discrete metric space, one usually takes $\cH = \ell^2(M)$.

By the spectral theorem, the $C_0(M)$ action on $\cH$ extends canonically to a $*$-representation of the algebra of bounded Borel functions on $M$ in the algebra $\mathscr{L}(\mathcal{H})$ of bounded operators on $\mathcal{H}$.
For any bounded Borel function $f$ on $M$, we denote the operator on $\mathcal{H}$ that multiplies with $f$ (using the module action) by the same letter.
Moreover, for any Borel subset $X \subseteq M$, we continue to denote its indicator function and, consequently, the corresponding multiplication operator by the same letter.

\begin{definition}[Finite propagation, locally trace-class operators]
\label{DefLTCFP}
Let $\cH$ be an $M$-module and let $T \in \sL(\cH)$.
\begin{enumerate}[(i)]
\item 
$T$ has \emph{propagation bound} $R$, if whenever the supports of two bounded Borel functions $f$ and $g$ have distance at least $R$, then $f T g =0$;
\item 
$T$ has \emph{finite propagation}, if it admits some finite propagation bound;
\item 
$T$ is \emph{locally trace-class}, if $fT$ and $Tf$ are trace-class operators for any Borel function $f$ with bounded support.
\end{enumerate}

 We denote by $\sB(M) \subseteq \sL(\cH)$ the subalgebra of locally trace-class, finite propagation operators. 

\begin{enumerate}
	\item[(iv)]
	If $Y \subseteq M$, we say that an operator $T\in\sL(\mathcal{H})$ is \emph{supported on} $Y$ if $fT = Tf = 0$ whenever $f$ is a bounded Borel function with $\supp(f) \subseteq Y^c$;
	\item[(v)]
	If $\cY$ is a big family in $M$, we say that $T$ is \emph{supported on} $\cY$ if it is supported on some member of $\cY$.
\end{enumerate}

We denote by $\sB(\mathcal{Y}) \subseteq \sB(M) \subset \sL(\cH)$ the subalgebra consisting of those locally trace-class and finite propagation operators which are supported on $\mathcal{Y}$. It will follow from Lemmas \ref{LemmaTSsupported} and \ref{LemmaLocallyTraceclass} below that the subalgebras $\sB(\cY)$ are ideals in $\sB(M)$, so we refer to them as \emph{localization ideals}.
\end{definition}

\begin{remark}\label{RemAmpleModule}
We may write $\sB_{\cH}(\cY)$ to emphasize the dependence on the choice of $M$-module, but will usually suppress $\cH$ in notation. 
This simplified notation has the justification that if $\cH$ and $\cH'$ are \emph{ample}---meaning that no non-zero $f \in C_0(M)$ acts as a compact operator---then $\sB_{\cH}(\cY) \cong \sB_{\cH'}(\cY)$ for any big family $\cY$ in $M$, where the isomorphism is provided by conjugation by a finite propagation unitary $\cH \to \cH'$.
In particular, the algebras $\sB_{\cH}(\cY)$ have canonically isomorphic $K$-theory.
Similar remarks hold for the norm closures $C^*_{\cH}(\cY)$ of $\sB_{\cH}(\cY)$, which are called the \emph{Roe algebras} associated to the big families $\cY$. 
For a reference on these facts, see \cite[\S6.3]{HigsonRoe}.
\end{remark}

\begin{lemma}
\label{LemmaTSsupported}
If $S, T \in \sL(\cH)$ are such that $T$ has finite propagation and $S$ is supported on $Y \subseteq M$, then $TS$ and $ST$ are supported on $\{Y\}$.
\end{lemma}

\begin{proof}
Let $R$ be a propagation bound for $T$.
Let $f$ be a bounded Borel function with $\supp(f) \subseteq (Y_R)^c \subseteq Y^c$.
Then clearly $fST = TSf = 0$.
For the other products, observe that $S = SY = YS$ because $S(1-Y) = (1-Y)S = 0$ from the fact that $S$ is supported on $Y$.
Hence $STf = SYTf = 0$ and similarly $fTS = fTYS = 0$ because $d(\supp(f), Y) \geq R$, the propagation bound of $T$.
\end{proof}

\begin{lemma}
\label{LemmaLocallyTraceclass}
If $S, T \in \sL(\cH)$ are such that $T$ has finite propagation and $S$ is locally trace-class, then also $ST$ and $TS$ are locally trace-class.
\end{lemma}

\begin{proof}
Let $f$ be a bounded Borel function supported in a bounded subset $B$. 
Then clearly $fST$ and $TSf$ are trace-class.
By Lemma~\ref{LemmaTSsupported}, $Tf$ and $fT$ are supported on $B_R$ for some $R\geq 0$ (one can choose $R$ to be a propagation bound for $T$).
Hence $STf = SB_RTf$ and $fTS = fTB_RS$ are trace-class, as $SB_R$ and $B_RS$ are trace-class.
\end{proof}

\begin{lemma}
\label{LemmaTraceClass}
The algebra $\sB(\cB)$ consists of trace-class operators.	
\end{lemma}

\begin{proof}
Suppose that $T$ is a locally trace-class operator that is supported on a bounded subset $B \subseteq M$. 
Then we may choose a Borel function $f$ with bounded support and $f_{|B} \equiv 1$.
Then $T = Tf + T(1-f)$ and because $\supp(1-f) \subseteq B^c$ while $T$ is supported on $B$, we get that $T(1-f) =0$.
So $T = Tf$, which is trace-class.
\end{proof}

\begin{lemma}
\label{LemmaLocalizationIdealProperty}
Let $Y_1, \dots, Y_n \subseteq M$ be coarsely transverse and let $T_1, \dots, T_n$ be finite propagation operators, at least one of which is locally trace-class.
Then the product 
\[
Y_1T_1 \cdots Y_n T_n \quad \text{is trace-class.}
\]
\end{lemma}

\begin{proof}
Since one of the operators is locally trace-class and each of the $T_i$ as well as the multiplication operators $Y_i$ has finite propagation, the product is locally trace-class by Lemma~\ref{LemmaLocallyTraceclass}.

By Lemma~\ref{LemmaTSsupported}, for each $i=1, \dots, n$, the operators $Y_i T_i$ are supported on $\{Y_i\}$.
Fixing $i$, using the same lemma iteratively shows that the whole product is supported on $\{Y_i\}$.
Altogether this implies that $Y_1T_1 \cdots Y_n T_n$ is supported on $\{Y_1\} \Cap \cdots \Cap \{Y_n\}$.
By coarse transversality, this is the big family $\cB$ of bounded subsets, so the product is trace-class by Lemma~\ref{LemmaTraceClass}.
\end{proof}

We will frequently use the unitization $\sB(M)^+$ of the algebra of finite propagation locally trace-class operators, which may be realized as the subalgebra of $\mathscr{L}(\cH)$ consisting of those operators that differ from $\sB(M)$ by a multiple of the identity operator.

For $k\geq 1$, the matrix algebras $M_k(\sB(M)^+)$ are regarded as acting on the $M$-module $\mathcal{H}^{\oplus k}$,
\[
M_k(\sB(M)^+) \subseteq \sL(\cH^{\oplus k}).
\]
Elements $T\in M_k(\sB(M)^+)$ have the form $T=T_0 + S$ where $S\in M_k(\sB(M))$ while $T_0$ is scalar in the sense that its entries are scalar multiples of the identity operator on $\mathcal{H}$.
We write $\GL_k(\sB(M))^+$ for the group of invertible elements of $M_k(\sB(M)^+)$ whose scalar part is the identity operator on $\mathcal{H}^{\oplus k}$.


\subsection{Kitaev and Kubo pairings}\label{SecKitaevKuboPairings}

Let $M$ be a proper metric space and let $\sB(M)$ be the corresponding algebra of locally trace-class, finite propagation operators, defined on some $M$-module $\cH$.

\begin{definition}[Kitaev and Kubo pairings]
\label{DefinitionIdempotentPairing}
Let $A_0, \dots, A_n$ be a coarsely transverse partition of $M$. We define its \emph{Kitaev pairings} by the formulas
\begin{center}
\begin{tcolorbox}[width=0.87\textwidth,colback=white, colframe=black, sharp corners, boxrule=0.55pt]
\vspace{-0.3cm}
\begin{align*}
	\big[ P; A_0, \dots, A_n\big] &:= \Tr\Bigg(\sum_{\sigma \in S_{n+1}} \sgn(\sigma) \cdot A_{\sigma_0} P A_{\sigma_1} P\cdots A_{\sigma_n} P\Bigg)
	\\
	\big[ U; A_0, \dots, A_n\big] &:= \Tr\Bigg(\sum_{\sigma \in S_{n+1}} \sgn(\sigma) \cdot A_{\sigma_0} U A_{\sigma_1} U^{-1}\cdots A_{\sigma_{n-1}} U A_{\sigma_n} U^{-1}\Bigg)
\end{align*}
\end{tcolorbox}
\end{center}
where $P = P^2 \in M_k(\sB(M)^+)$ is an idempotent, $U \in \GL_k(\sB(M))^+$ is invertible, and $S_{n+1}$ denotes the group of permutations of $\{0,1,\ldots,n\}$.
In the second formula, we assume that $n$ is odd.
For a coarsely transverse collection of half-spaces $X_1, \dots, X_n$, we set $X_0 := \mathbf{1}$ and define its \emph{Kubo pairings} by
\begin{center}
\begin{tcolorbox}[width=0.9\textwidth,colback=white, colframe=black, sharp corners, boxrule=0.5pt]
\vspace{-0.3cm}
\begin{align*}
	\big[ P; X_1, \dots, X_n\big] &:= \Tr\Bigg(\sum_{\sigma \in S_{n+1}} \sgn(\sigma)\cdot X_{\sigma_0} P X_{\sigma_1} P \cdots X_{\sigma_n} P\Bigg)
	\\
	\big[ U; X_1, \dots, X_n\big] &:= \Tr\Bigg(\sum_{\sigma \in S_{n+1}} \sgn(\sigma)\cdot X_{\sigma_0}U X_{\sigma_1} U^{-1} \cdots X_{\sigma_{n-1}} U X_{\sigma_n} U^{-1}\Bigg),
\end{align*}
\end{tcolorbox}
\end{center}
where in the last formula, we again assume that $n$ is odd.
\end{definition}
The well-definedness of the traces in Definitions \ref{DefinitionIdempotentPairing} is verified by a hands-on approach in Section \ref{sec:WellDefinedPairings}. 
Alternatively, it also follows from a more abstract analysis of the pairings in Section \ref{SectionCyclicHomology}.

\begin{remark}\label{RemSwitchFunctions}
More generally, one may use a collection $\chi_1, \dots, \chi_n$ of coarsely transverse switch functions in the definition of the Kubo pairing.
However, we will see below (Lemma~\ref{LemSwitchToHalfSpace}) that one may always replace the $\chi_i$ with the indicator functions of their support (thus obtaining a coarsely transverse collection of half-spaces) without changing the corresponding Kubo pairing. 
Because of this, we will mostly work with half-spaces $X_1,\ldots, X_n$ to simplify the presentation.
\end{remark}

\begin{remark}
\label{remarkn=0}
An elaboration is needed for the Kitaev/Kubo pairing when $n=0$.
 Only compact spaces $M$ admit $0$-partitions, in which case $A_0=M=X_0$, and $\sB(M)$ is the algebra $\sL_1$ of trace-class operators on $\mathcal{H}$. 
 By convention, $\Tr$ refers to the trivial extension of the usual trace to $\sB(M)^+=\sL_1^+$ given by $\smash{\widetilde{\Tr}(T + \lambda)} = \Tr(T)$, and likewise to $M_k(\sB(M)^+)$. 
 Thus, the Kitaev pairing of $P=P^2\in M_k(\sB(M)^+)$ with $A_0=M$, or the Kubo pairing with the empty collection of half-spaces, is equal to $\smash{\widetilde{\Tr}(P)} = \Tr(Q)$, where $Q\in M_k(\sB(M))$ is the non-scalar part of $P$.
\end{remark}

\begin{remark}
The $n=1$ case with $M=\Z$, $X_1=\N$, and $U=(U_{jk})_{j,k\in\Z}$ a finite propagation unitary on the $M$-module $\cH = \ell^2(\Z;\C )$, was considered by Kitaev in \cite[\S C.1]{Kitaev}. 
Here we have
\begin{equation}\label{eqn:Kitaev_1D_formula}
[U;X_1]=\Tr(UX_1U^*-X_1) 
=\sum_{j\geq 0, k<0}\big(|U_{kj}|^2-|U_{jk}|^2\big),
\end{equation}
which was called the \emph{flow} of $U$. (Our sign convention differs from that in \cite{Kitaev}.) 
Such unitaries are sometimes called \emph{quantum walks}, and \eqref{eqn:Kitaev_1D_formula} was used in \cite{GNVW} as an index for studying quantum walks. When endowed with a suitable topology, the homotopy groups of the group of finite propagation unitary operators on $\ell^2(\mathbb{Z};\mathbb{C})$ were studied in \cite{KKT}.
\end{remark}

\begin{remark}
The $n=2$ and $M=\Z^2$ case of the Kitaev pairing, Definition~\ref{DefinitionIdempotentPairing}, was considered by Kitaev in \cite[\S C.3]{Kitaev}, as an alternative formula for (half of) the Hall conductance of a projection $P$, (up to a $2\pi i$ factor, Planck's constant and the electron charge).
A more commonly used formula is the Kubo pairing, which has an equivalent expression in terms of commutators with switch functions (see \eqref{KuboPairingCommutatorProduct}),
which is
\begin{equation*}
\Tr\Big(P \big[[P,\chi_1], [P,\chi_2]\big]\Big),
\end{equation*}
see, e.g., \cite{DEF,EGS}. Typically, one works with $M=\R^2$ and uses switch functions of the form $\chi_i(x_1,x_2)=\Lambda(x_i), i=1,2$, where $\Lambda:\R\to\R$ satisfies
\[
\Lambda(t)=\begin{cases} 1, & t>b,\\
0, &t<a,
\end{cases}
\]
for some interval $[a,b]$. Our more general notion of switch functions emphasizes that coarse transversality of their supports is the key underlying property.

In physics, a typical source of $P$ are the spectral projections for some spectrally gapped local Hamiltonian operator $H$. Such $P$ are usually locally trace-class, but do not have strict finite propagation. 
With sufficiently rapid decay properties, the Kubo formula continues to be well-defined for such $P$, see \cite{LudewigThiangPoly}. 
An example on $M=\R^2$ is explicitly computed in \cite{ThiangFock}.
\end{remark}

The main result of this paper is the following general quantization theorem for the higher Kubo/Kitaev trace pairings, which will be proved in Sections~\ref{SectionIntegralityKubo} and \ref{SectionIntegralityKitaev}.

\begin{theorem}
\label{MainQuantizationThm}
Let $A_0, \dots, A_n$ be a coarsely transverse partition of $M$, and $X_1, \ldots, X_n$ be a  coarsely transverse collection of half-spaces on $M$.

\noindent
If $n=2m$ is even, the pairings with an idempotent $P\in M_k(\sB(M)^+)$ satisfy
\begin{align*}
 \big[ P; A_0, \dots, A_{n}\big] &\in \frac{m!}{(2 \pi i)^{m} (2m)!} \cdot \Z,\\
\big[ P; X_1, \dots, X_{n}\big] &\in \frac{m!}{(2 \pi i)^{m} }  \cdot \Z.
\end{align*}
If $n=2m+1$ is odd, the pairings with invertibles $U\in \GL_k(\sB(M))^+$ satisfy
\begin{align*}
  \big[ U; A_0, \dots, A_{n}\big] & \in \frac{1}{(2 \pi i)^{m} m! } \cdot \Z,	\\
  \big[ U; X_1, \dots, X_{n}\big] & \in \frac{(2m +1)!}{(2 \pi i)^{m}m!} \cdot \Z.
\end{align*}
\end{theorem}

We will see below that the pairings with $P$ vanish when $n$ is odd.
For $M= \R^d$, the values of the pairings for $n=d$ are not contained in any smaller subset than the ones specified above, so this result is optimal for general $M$, see Thm.~\ref{ThmNontrivialPairing}.

The following result shows that only the even $n$ case can produce non-trivial pairings with idempotents.

\begin{proposition}
\label{ThmVanishingThmRemainingCases}
Let $A_0, \dots, A_n$ be a coarsely transverse partition of $M$ and let $X_1, \dots, X_n$ be a coarsely transverse collection of half-spaces. 
If $n$ is odd, then for any idempotent $P \in M_k(\sB(M)^+)$, we have
\[
[P; A_0, \dots, A_n] = 0 \qquad \text{and} \qquad [P; X_1, \dots, X_n] = 0.
\] 
\end{proposition}

\begin{proof}
By cyclicity of the trace, we have
\begin{align*}
[P; A_0, \dots, A_n]
&= \sum_{\sigma \in S_{n+1}} \sgn(\sigma) \cdot \Tr\big(A_{\sigma_0} P \cdots A_{\sigma_{n-1}}PA_{\sigma_n}P\big)	
\\
&= \sum_{\sigma \in S_{n+1}} \sgn(\sigma) \cdot \Tr\big(A_{\sigma_n}PA_{\sigma_0} P \cdots A_{\sigma_{n-1}}P\big).
\end{align*}
Since $n$ is odd, the cyclic permutation of $n+1$ elements has negative sign. 
Hence substituting $\sigma'_0 = \sigma_n$ and $\sigma_{i+1}' = \sigma_i$, (for $i= 0, \dots, n-1$), we get $\sgn(\sigma') = -\sgn(\sigma)$ and therefore the right hand side above equals $-[P; A_0, \dots, A_n]$. 
Thus $[P; A_0, \dots, A_n]=0$.

Using cyclicity of the trace and the fact that $P$ is an idempotent, we may also take the sum over $PX_{\sigma_0}P \cdots P X_{\sigma_n} P$ in the definition of the Kubo pairing (Definition~\ref{DefinitionIdempotentPairing}), which is therefore equal to the trace of the following operator,
\begin{align*}
\sum_{\sigma\in S_{n+1}}\sgn(\sigma)\cdot PX_{\sigma_0} P \cdots PX_{\sigma_n}P 
&=  \sum_{i=0}^n \sum_{\substack{\sigma\in S_{n+1} \nonumber\\ \sigma_i=0}}\sgn(\sigma)\cdot PX_{\sigma_0} P \cdots \underbrace{PX_{\sigma_i} P}_{=P} \cdots PX_{\sigma_n}P
 \nonumber \\
 &= \sum_{i=0}^n\sum_{\sigma^\prime\in S_{n}}(-1)^{i}\cdot\sgn(\sigma^\prime)\cdot PX_{\sigma^\prime_1} P \cdots PX_{\sigma^\prime_{n}} P. 
 \end{align*}
If $n$ is odd, then $\sum_{i=0}^n (-1)^i = 0$, so the above operator vanishes. Thus the Kubo pairing vanishes in this case.
\end{proof}

\subsection{Well-definedness of the pairings}\label{sec:WellDefinedPairings}

In this section, we show that the Kitaev and Kubo pairings from Definition~\ref{DefinitionIdempotentPairing} are well-defined in the sense that the operators that the trace is taken of are in fact trace-class.
We also derive several alternative formulas for these pairings.

\paragraph{Kitaev pairing with idempotents.}

It follows immediately from Lemma~\ref{LemmaLocalizationIdealProperty} above that if $P$ is an idempotent that is locally trace-class (i.e., contained in $M_k(\sB(M)) \subset M_k(\sB(M)^+)$), then each of the operators $A_{\sigma_0} P \cdots A_{\sigma_n} P$ is trace-class, so the Kitaev pairing is well-defined and we have
\[
[P; A_0, \dots, A_n] = \sum_{\sigma \in S_{n+1}}\sgn(\sigma) \cdot \Tr \big(A_{\sigma_0} P \cdots A_{\sigma_n} P\big),
\]
in other words, the sum over all permutations may be pulled out of the trace.
However, the pairings will be trivial for a large class of such $P$, hence it is especially important to check well-definedness also for projections $P$ in $M_k(\sB(M)^+)$, see Theorem \ref{ThmNontrivialPairing} and Remark \ref{RemarkTranslationInvariantVanish}.
In this case, we may write $P=P_0+Q$, where $P_0\in M_k(\C)$ is a scalar matrix (whose entries are multiples of $1_\mathcal{H}$) while the remainder $Q$ lies in $M_k(\sB(M))$.
As $P_0$ commutes with every $A_i$, we have $A_{\sigma_i}P_0A_{\sigma_{i+1}}=P_0A_{\sigma_i}A_{\sigma_{i+1}}$. Therefore, any term with a factor $A_{\sigma_i}P_0A_{\sigma_{i+1}}$ will vanish after taking the signed sum over all permutations.
Thus, 
\begin{align*}
&\sum_{\sigma \in S_{n+1}}^{\,} \sgn(\sigma)\cdot A_{\sigma_0} P \cdots A_{\sigma_n} P 
=\sum_{\sigma \in S_{n+1}}^{\,} \sgn(\sigma)\cdot A_{\sigma_0} Q \cdots A_{\sigma_{n-1}} Q A_{\sigma_n} (Q+P_0)\nonumber\\
&\qquad =\sum_{\sigma \in S_{n+1}}^{\,} \sgn(\sigma)\cdot A_{\sigma_0} Q \cdots A_{\sigma_n}Q 
 + \sum_{\sigma \in S_{n+1}}^{\,} \sgn(\sigma) \cdot A_{\sigma_0} Q \cdots A_{\sigma_{n-1}} Q A_{\sigma_n}P_0.
\end{align*}
As $Q\in M_k(\sB(M))$ has finite propagation and is locally trace-class, Lemma~\ref{LemmaLocalizationIdealProperty} applies, saying that each individual summand is trace-class.
Because $P_0$ commutes with $A_{\sigma_n}$, the second term actually vanishes after applying the trace and using its cyclicity.
Therefore, the Kitaev pairing reduces to
\begin{equation}
\label{AlternativeFormulaKitaevPairing}
\big[ P; A_0, \dots, A_n\big] =\sum_{\sigma \in S_{n+1}} \sgn(\sigma) \cdot \Tr\big( A_{\sigma_0} Q \cdots A_{\sigma_n} Q\big).
\end{equation}
Note that the $n=0$ case of \eqref{AlternativeFormulaKitaevPairing} is precisely the convention for the $n=0$ Kitaev/Kubo pairing mentioned in Remark \ref{remarkn=0}.
%
%

\paragraph{Kitaev pairing with invertibles.}
Because $U$ and $U^{-1}$ lie in $\GL_k(\sB(M))^+$, the operators $U-1$ and $U^{-1}-1$ belong to $M_k(\sB(M))$, hence are locally trace-class. 
With similar arguments as above, one may therefore establish the alternative formula
\begin{equation}
\label{KitaevPairingAlt}
[U; A_0, \dots, A_n] = 	\sum_{\sigma \in S_{n+1}}^{\,} \sgn(\sigma) \cdot \Tr\big( A_{\sigma_0} (U - 1) A_{\sigma_1} (U^{-1} - 1)\cdots A_{\sigma_n} (U^{-1} - 1)\big),
\end{equation}
for which well-definedness follows again from Lemma~\ref{LemmaLocalizationIdealProperty}.
%

\paragraph{Kubo pairing with idempotents.}

For $n=0$, see Remark \ref{remarkn=0}; we now assume $n\geq 1$. 
Splitting $P = P_0 + Q$ with $P_0 \in M_k(\C)$ and $Q$ locally trace-class and finite propagation, we may argue as in the partition case to obtain the formula
\begin{equation}
\label{KuboPairingCommutatorProduct2}
  	\big[ P; X_1, \dots, X_n\big]
  	= \Tr\left(\sum_{\sigma \in S_{n+1}} \sgn(\sigma)\cdot X_{\sigma_0} Q X_{\sigma_1} Q \cdots X_{\sigma_n} Q\right).
\end{equation}
However, it is not apparent from this formula that the operator in the bracket is trace-class as Lemma~\ref{LemmaLocalizationIdealProperty} only yields that it is supported on the ``corner space'' $\{X_1\} \Cap \cdots \Cap \{X_n\}$ which is typically not bounded.
We will therefore show that the Kubo pairing may be rewritten as
\begin{equation}
\label{KuboPairingCommutatorProduct}
  	\big[ P; X_1, \dots, X_n\big]
  	= \sum_{\sigma \in S_{n}} \sgn(\sigma)\cdot \Tr\big(P[P,X_{\sigma_1}] \cdots [P,X_{\sigma_n}]\big).
\end{equation}
Then using that $X_{\sigma_i}$ commutes with $P_0$, we may replace each 
$[P,X_{\sigma_i}]$ in \eqref{KuboPairingCommutatorProduct} by $[Q,X_{\sigma_i}]$ and Lemma~\ref{LemmaLocallyTraceclass} implies that the product is locally trace-class.
Moreover, for each $i$, the commutator
\[
[Q,X_i]=(X_i+X_i^c)QX_i-X_iQ(X_i+X_i^c)=X_i^cQX_i-X_iQX_i^c
\]
is supported on $\{X_i\} \Cap \{X_i^c\} = \partial X_i$ by Lemma~\ref{LemmaTSsupported}.
The product over $i$ of these commutators is therefore supported on $\partial X_1 \Cap \cdots \Cap \partial X_n$, and is trace-class by Lemma~\ref{LemmaTraceClass} and the coarse transversality assumption.

The derivation of \eqref{KuboPairingCommutatorProduct} is a rather standard calculation, see \cite[Lemma~1.8]{MoscoviciWu2}: Recalling that $X_0=\mathbf{1}$, we first rewrite
\begin{align*}
&\sum_{\sigma\in S_{n+1}}\sgn(\sigma)\cdot X_{\sigma_0} P \cdots PX_{\sigma_n}P \\
&= \sum_{\substack{\sigma\in S_{n+1} \\ \sigma_0=0}}\sgn(\sigma)\cdot PX_{\sigma_1}P\cdots PX_{\sigma_n}P + \sum_{i=1}^n \sum_{\substack{\sigma\in S_{n+1} \nonumber\\ \sigma_i=0}}\sgn(\sigma)\cdot X_{\sigma_0} P \cdots \underbrace{PX_{\sigma_i} P}_{=P} \cdots PX_{\sigma_n}P\\
&=\sum_{\sigma^\prime\in S_{n} }\sgn(\sigma^\prime)\cdot PX_{\sigma^\prime_1}P\cdots PX_{\sigma^\prime_n}P + \sum_{i=1}^n \sum_{\sigma^\prime\in S_{n}}(-1)^i\cdot\sgn(\sigma^\prime)\cdot X_{\sigma^\prime_1} P \cdots PX_{\sigma^\prime_n}P,
\end{align*}
where in the last equality, for each $\sigma \in S_{n+1}$ with $\sigma_i=0$, (viewed as the group of permutations of the set $\{0, \dots, n\})$,  $\sigma'$ is the permutation of $\{1, \dots, n\}$ given by
\[
\sigma' = 
\begin{pmatrix}
1 & \cdots &i & i +1  & \cdots & n
\\
\sigma_0 & \cdots & \sigma_{i-1} & \sigma_{i+1} &\cdots & \sigma_n
\end{pmatrix},
\]
which has sign $\sgn(\sigma') = (-1)^i \cdot  \sgn(\sigma)$. 
As $\sum_{i=1}^n (-1)^i$ equals $0$ when $n$ is even and $-1$ when $n$ is odd, we have
\[
\sum_{\sigma\in S_{n+1}}\sgn(\sigma)\cdot X_{\sigma_0} P \cdots PX_{\sigma_n}P=
\sum_{\sigma\in S_n}\sgn(\sigma)\cdot\begin{cases}
PX_{\sigma_1} P\cdots PX_{\sigma_n} P, & n\;\mathrm{even},\\
 [P,X_{\sigma_1}] PX_{\sigma_2}P\cdots PX_{\sigma_n} P, & n\;\mathrm{odd}.
\end{cases}
\]
Next, we iteratively substitute the identity
\[
PX_iPX_jP  = [P,X_i][P,X_j]P + PX_iX_jP,
\]
noting that the term $PX_iX_jP$ may be omitted as it does not contribute to a sum which is anti-symmetric in $i,j$. In other words, 
\begin{equation*}
\sum_{\sigma\in S_{n+1}}\sgn(\sigma)\cdot X_{\sigma_0} P \cdots PX_{\sigma_n}P
=\sum_{\sigma\in S_n}\sgn(\sigma)\cdot [P,X_{\sigma_1}]\cdots [P,X_{\sigma_n}] P, 
\end{equation*}
whether $n$ is even or odd, which implies \eqref{KuboPairingCommutatorProduct} after taking the trace.

\paragraph{Kubo pairing with invertibles.}

As above, well-definedness of the pairing follows from the alternative formula
\begin{equation}\label{AlternativeFormulaKuboOdd1}
[U; X_1, \dots, X_n] = 	\sum_{\sigma\in S_{n}}\sgn(\sigma)\cdot \Tr\big(U[X_{\sigma_1},U^{-1}]\cdots [X_{\sigma_{n-1}},U][X_{\sigma_n},U^{-1}]\big),
\end{equation}
in which each summand is trace-class.
It will be important below that besides \eqref{AlternativeFormulaKuboOdd1}, one has another alternative formula
\begin{equation}
\label{AlternativeFormulaKuboOdd2}
[U; X_1, \dots, X_n] = 	\Tr\Bigg(\sum_{\sigma\in S_{n}}\sgn(\sigma)\cdot (U-1)X_{\sigma_1}(U^{-1}-1)\cdots X_{\sigma_n}(U^{-1}-1)\Bigg),
\end{equation}
which follows from the fact that when expanding the product for each individual summand, the terms involving an identity operator will vanish by anti-symmetry.

\subsection{Comparison to the theory of Connes}
\label{ConnesTheory}

Formula \eqref{KuboPairingCommutatorProduct} resembles well-known formulas for the Chern character of Fredholm modules in noncommutative geometry.
One might therefore wonder whether our quantization results are just straightforward generalizations of the noncommutative index theory of Connes \cite{Connes}; see also \cite{DouglasVoiculescu1981} for similar arguments.
In this section, we argue that this cannot be the case.

Recall that an even $p$-summable \emph{Fredholm module} over a (trivially graded) algebra $A$ is a super Hilbert space $\hat{\mathcal{H}}$, together with a representation of $A$ (as even operators in $\sL(\hat{\mathcal{H}})$) and an odd involution $F$ on $\hat{\mathcal{H}}$,  such that $[F, a]$ belongs to the $p$-Schatten ideal $\sL_p(\hat{\mathcal{H}})$ for all $a \in A$.
If the Fredholm module is $n$-summable, $n = 2m$ even, its $n$-th \emph{Chern character} is defined by 
\[
\tau_n(a_0, \dots, a_n) := 
\Str\big(a_0 [F, a_1] \cdots [F, a_n]\big),
\]
where $\Str$ denotes the supertrace.
This is a cyclic cocycle on $A$ \cite[\S2, Prop.~5]{Connes}.
%
%
If $P$ is an idempotent over $A$, defining a class in the  $K$-theory group $K_0^{\mathrm{alg}}(A)$, then the noncommutative index theorem of Connes says that
\[
\ind(PFP + 1-P) = \tau_n(\underbrace{P, \dots, P}_{n+1}) = \Str\big(P [F, P]^{n}\big),
\]
see \cite[\S3, Thm.~1]{Connes}.

In our case, we might try to apply this result as follows: For $n=2m\geq 2$, let $X_1, \dots, X_n$ be a coarsely transverse collection of half-spaces. 
We may choose odd self-adjoint Clifford generators $\Gamma_1, \dots, \Gamma_n$ acting on $\C^{2^{m-1}}\oplus \C^{2^{m-1}}$ 
 (i.e., satisfying $\Gamma_i\Gamma_j + \Gamma_j\Gamma_i = 2 \delta_{ij}$), and consider the operator
\begin{equation*}
F := \frac{1}{\sqrt{n}}\sum_{i=1}^n (X_i - X_i^c) \otimes \Gamma_i
\end{equation*}
acting on $\hat{\mathcal{H}} := \mathcal{H} \otimes (\C^{2^{m-1}}\oplus \C^{2^{m-1}})$, where $\mathcal{H}$ is an $M$-module.
The Clifford generators can be chosen such that the supertrace of the Clifford volume element $\Gamma_1 \cdots \Gamma_n$ is $(2i)^m$, while that of $\Gamma_{i_1} \cdots \Gamma_{i_n}$ vanishes whenever any of the indices $i_1, \dots, i_n$ is repeated. 
One then calculates that
\begin{equation}
\Str\big(a_0 [F, a_1] \cdots [F, a_n]\big)
 = \left(\frac{4i}{m}\right)^m \cdot \sum_{\sigma \in S_n} \sgn(\sigma) \cdot \Tr\big(a_0[X_{\sigma_1}, a_1] \cdots [X_{\sigma_n}, a_n]\big),\label{eqn:SuperTrace}
\end{equation}
under the assumption that trace-class conditions hold throughout the computation.
%
%
Taking $a_0=\dots=a_n=P$ to be an idempotent in $M_k(\sB(M)^+)$, the last expression in \eqref{eqn:SuperTrace} is just a multiple of the Kubo trace pairing \eqref{KuboPairingCommutatorProduct}.
So combining \eqref{eqn:SuperTrace} with the definition of the Chern character and the index theorem of Connes, we would obtain for the trace pairing
\[
\big[P; X_1, \dots, X_n\big] = \left(\frac{m}{4i}\right)^m \cdot \tau_n(\underbrace{P, \dots, P}_{n+1}) \in \left(\frac{m}{4i}\right)^m \cdot  \Z,
\]
contradicting Thm.~\ref{MainQuantizationThm}.
The point is that the operator $[F, P]$ will \emph{not} generally be in $\sL_p(\hat{\mathcal{H}})$ for any $p \in [1, \infty)$, so the left side of \eqref{eqn:SuperTrace} is ill-defined and Connes' index theorem cannot be applied. Specifically, the operators $a_0 [X_{i_1}, a_1] \cdots [X_{i_n}, a_n]$ are only guaranteed to be trace-class when none of the indices $i_1,\ldots,i_n$ are repeated (see Lemma~\ref{LemmaLocalizationIdealProperty}), so $(\hat{\mathcal{H}},F)$ is not actually an $n$-summable Fredholm module.

What seems to be going on here is that in Connes' theory for Fredholm modules, the operator ideals $\sL_p(\cH)$ being considered are \emph{totally} ordered.
In contrast, in our theory of Kubo trace pairings, the products $[X_{i_1}, P]\cdots [X_{i_\ell}, P]$ are contained in the ideal of operators localized at the intersection of the boundaries of $X_{i_1} , \dots, X_{i_\ell}$ (see Lemma~\ref{LemmaLocalizationIdealProperty} and the proof of Eq.~\eqref{KuboPairingCommutatorProduct}).
These localization ideals have a partial order by inclusion, which is \emph{not} a total order.

\section{Coarse (co)homology}
\label{SectionCoarseCohomology}

In this section, we give the necessary background on coarse homology and cohomology and introduce the coarse cochains induced by partitions and half-spaces.

Coarse cohomology was introduced by J.~Roe for the study of indices of Dirac operators on non-compact complete Riemannian manifolds \cite{Roe-cohom}. 
For the purposes of this paper, we provide a presentation of the dual coarse homology theory in terms of measures, which is particularly well-suited for the case of proper metric spaces; 
see also \cite{Mitchener,WulffCoarseCohomology} for slightly different pictures.

\subsection{Definition of coarse homology}

For each $n\geq 0$, we equip $M^{n+1}$ with the distance function that takes the maximum of the distances within each factor. 
The \emph{multi-diagonal} is
\[
\Delta = \{(x, \dots, x) \in M^{n+1} \mid x \in M\}.
\]
A \emph{coarse $n$-chain on} $M$ is a complex-valued, locally finite regular Borel measure on $M^{n+1}$, which has support in a thickening $\Delta_R$ of the diagonal for some $R>0$. 
We denote the space of coarse $n$-chains by $CX_n(M)$. The spaces $CX_\bullet(M)$ form a chain complex, with the differential
\begin{equation*}
\partial \mu = \sum_{i=0}^n (-1)^i (\pi_i)_* \mu,\qquad \mu\in CX_n(M),
\end{equation*}
where $\pi_i : M^{n+1} \to M^n$ $(0 \leq i \leq n)$ is the projection omitting the $i$-th factor.
The corresponding homology groups are called the \emph{coarse homology groups} of $M$.

We say that a measure $\mu \in CX_n(M)$ is \emph{supported on $Y \subseteq M$} if it vanishes outside of $Y^{n+1}$.
If $\cY$ is a big family in $M$, we say that $\mu$ is \emph{supported on} $\cY$ if it is supported on $Y$ for some member $Y$ of $\cY$ and denote by  $CX_\bullet(\cY) \subseteq CX_\bullet(M)$ the subcomplex of measures  supported on $\cY$. 

\begin{definition}[Coarse homology]
The homology groups of the chain complex $CX_\bullet(\cY)$ are called \emph{coarse homology} groups of $\cY$ and denoted by $HX_\bullet(\cY)$.
If $\cX$ is a further big family, the \emph{relative} coarse homology groups $HX_\bullet(\cX, \cX \Cap \cY)$ are the homology groups of the quotient complex $CX_\bullet(\cX)/CX_\bullet(\cX\Cap\cY)$.
\end{definition} 

\begin{remark}
For any big family $\cY$, the space $CX_n(\cY)$ is the filtered colimit of all $CX_n(Y)$, for $Y \in \cY$, where the $Y$ are viewed as metric spaces with the induced metric, and the connecting maps in the colimit are given by inclusions.
Because filtered colimits commute with taking homology, we obtain
\begin{equation}
\label{ColimitFormula}
HX_n(\cY) \cong \colim_{Y \in \cY} \,HX_n(Y).
\end{equation}
In particular, if $\cY = \{Y_0\}$ is generated by a single member, then all connecting maps are isomorphisms as soon as the members are large enough, by coarse invariance of coarse homology (see \cite[\S4.1]{WulffCoarseCohomology}).
Hence $HX_n(\cY) \cong HX_n(Y_0)$.
\end{remark}

A (not necessarily continuous) Borel-measurable map $f: M \to N$ between proper metric spaces is a \emph{coarse map} if it is \emph{controlled}, meaning that for each $r>0$, there exists $R>0$ such that $d(x, y) \leq r$ implies $d(f(x), f(y)) \leq R$, and \emph{proper}, meaning that the inverse image $f^{-1}(B)$ of each bounded set $B \subseteq N$ is bounded.

A controlled map $f : M \to N$ sends uniform neighborhoods of the multi-diagonal in $M$ to uniform neighborhoods of the multi-diagonal in $N$, while properness (and measurability) of $f$ ensures that pushforward along $f \times \cdots \times f$ sends locally finite measures to locally finite ones. 
Hence any coarse map induces (via pushforward along $f \times \cdots \times f$) a chain map
\[
f_* : CX_\bullet(M) \longrightarrow CX_\bullet(N).
\]
Coarse homology is \emph{coarsely invariant}, in the following sense.

\begin{proposition}
\label{PropCoarselyInvariant}
If two measurable coarse maps $f, g: M \to N$ are close, then they induce the same maps in coarse homology.
\end{proposition}

Recall here that two maps $f, g: M \to N$ are close if and only if there exists some $r >0$ such that $d(f(x), g(x)) \leq r$ for each $x \in M$.
A coarse map $f: M \to N$ is a \emph{coarse equivalence} if and only if there exists a coarse map $f' : N \to M$ such that both $f \circ f'$ and $f' \circ f$ are close to the identity. 
The above proposition shows in particular that coarse equivalences induce isomorphisms in coarse homology.

\begin{proof}
Define maps $h_n^i : M^n \longrightarrow N^{n+1}$ by
\[
h_n^i(x_0, \dots, x_{n-1}) = \big(f(x_0), \dots, f(x_i), g(x_i), \dots, g(x_{n-1})\big).
\]
Because $f$ and $g$ are close, $h_n^i$ sends a uniform neighborhood of the multi-diagonal $\Delta \subset M^n$ to some uniform neighborhood of the multi-diagonal $\Delta \subset N^{n+1}$, hence pushforward along each $h_n^i$ sends $CX_{n-1}(M)$ to $CX_{n}(N)$.
A chain homotopy between $f_*$ and $g_*$ is then given in degree $n$ by
\[
h_n := \sum_{i=0}^{n-1} (-1)^i \cdot (h_n^i)_*. \qedhere
\]
\end{proof}

\begin{remark}
In \cite{WulffCoarseCohomology}, the complex giving rise to coarse homology consists of functions vanishing outside some $\Delta_R$ that have locally finite support.
Such a function can be identified with a locally finite sum of Dirac measures, thus giving rise to an element in our complex. This defines a chain map $CX_\bullet^{\mathrm{Wulff}}(M) \to CX_\bullet(M)$, which is in fact a quasi-isomorphism.
This follows because both theories are coarsely invariant (by \cite[Lemma~4.9]{WulffCoarseCohomology} and Prop.~\ref{PropCoarselyInvariant}) and agree on discrete spaces, together with the fact that every proper metric space admits a discretization.
\end{remark}

We will use the following notion of excision in coarse homology.

\newpage

\begin{proposition}[Excision]
\label{PropExcision}
Let $\cX$ be a big family in $M$ and let $ \cZ, \cW$ be big families such that $\cX \Cap \cW \subseteq \cX \Cap \cZ$.
Then the canonical map
\[
HX_n(\cX, \cX \Cap \cZ) \longrightarrow HX_n\big(\cX \Cup \cW, (\cX \Cap \cZ) \Cup \cW\big)
\]
is an isomorphism.
\end{proposition}

\begin{proof}
The map is induced by the obvious map of quotient chain complexes
\begin{equation}
\label{QuotientMap}
\frac{CX_\bullet(\cX)}{CX_\bullet(\cX \Cap \cZ)} \longrightarrow \frac{CX_\bullet(\cX \Cup \cW)}{CX_\bullet((\cX\Cap\cZ) \Cup \cW)}
\end{equation}
descended from the inclusion $CX_\bullet(\cX)\to CX_\bullet(\cX \Cup \cW)$.
We claim that the map \eqref{QuotientMap} is an isomorphism.
To see surjectivity, let $\mu \in CX_n(\cX \Cup \cW)$ be supported on $X \cup W$ for $X \in \cX$, $W \in \cW$.
Then by a calculation similar to the proof of Prop.~\ref{PropMayerVietoris}, the measure
\begin{equation*}
\tilde{\mu} := (X \otimes \mathbf{1} \otimes \cdots \otimes \mathbf{1}) \cdot \mu
\end{equation*}
is contained in $CX_n(\cX)$ and the difference $\tilde{\mu} - \mu$ is contained in the subspace $CX_n(\cW) \subseteq CX_n((\cX \Cap \cZ)\Cup \cW)$.
Hence $\mu$ is the image of $\tilde{\mu}$ under the quotient map \eqref{QuotientMap}.

To see injectivity, notice that any element of $CX_\bullet(\cX)$ that is mapped to zero under \eqref{QuotientMap} is contained in
\[
CX_n(\cX) \cap CX_n\big((\cX \Cap \cZ)\Cup \cW\big) = CX_n\big(\cX \Cap \big((\cX \Cap \cZ) \Cup \cW\big)\big) \subseteq CX_n(\cX \Cap \cZ),
\]
where for the last inclusion, we used $\cX\Cap \cW \subseteq \cX\Cap \cZ$.
This shows that \eqref{QuotientMap} is injective.
\end{proof}

\begin{remark}
Via \eqref{ColimitFormula}, the above excision statement also implies that
\[
HX_n(\cX, \cX \Cap \cZ) \cong HX_n\big(\cX \cap W, (\cX \Cap \cZ) \cap W\big)
\]	
for  subsets $W\subseteq M$, where $\cX \cap W = \{X \cap W \mid X \in \cX\}$ etc., provided that the condition $\{W\}\Cup (\cX \Cap \cZ) = \{M\}$ holds.
\end{remark}

\subsection{Coarse cohomology and duality}

Dually, we may define coarse cohomology as follows. 
An \emph{$n$-cochain} on $M$ is a locally bounded Borel map $\theta:M^{n+1}\to \C$. 
There is a \emph{coboundary} map $\delta$ taking $n$-cochains to $(n+1)$-cochains, defined as
\begin{equation}
\label{ASDifferential}
\delta\theta = \sum_{i=0}^{n+1} (-1)^i\cdot  \pi_i^*\theta.
\end{equation}
If $\cY$ is a big family in $M$, we denote by $CX^n(\cY)$ the subspace of $n$-cochains $\theta$ that satisfy the following support condition: For each $R >0$ and each member $Y$ of $\cY$, the intersection
\[
  \supp(\theta) \cap \Delta_R \cap Y^{n+1} \qquad \text{is bounded}.
\]
Such cochains are called \emph{coarse on} $\cY$.
In particular, if $\cY = \{M\}$  we just write $CX^n(M)$ for the space of \emph{coarse $n$-cochains}, i.e., those locally bounded Borel functions whose support intersects any uniform neighborhood of the multi-diagonal in a bounded set.
For any big family $\cY$ in $M$, the spaces $\smash{CX^\bullet(\cY)}$ with the Alexander-Spanier differential \eqref{ASDifferential} form a cochain complex.

\begin{definition}[Coarse cohomology]
The cohomology groups of the above cochain complex are called the \emph{coarse cohomology groups} of $\cY$, denoted by $HX^\bullet(\cY)$.
\end{definition}

\begin{remark}
\label{RemarkLim1}
The relation of $HX^n(\cY)$ to the groups $HX^n(Y)$, $Y \in \cY$, is more subtle in the cohomology case.
Here the space $CX^n(\cY)$ can be written as a limit,
\[
CX^n(\cY) = \lim_{Y \in \cY} \, CX^n(Y),
\]
where the connecting maps are given by restriction of cochains.
%
%
Since all the restriction maps are surjective, the Mittag-Leffler condition is satisfied, hence the cohomology of $CX^n(\cY)$ can be calculated with Milnor's exact $\lim^1$-sequence
\[
\begin{tikzcd}[column sep=0.6cm]
0 \ar[r]
&
\displaystyle{\lim_{Y \in \cY}}^{\!1}\, HX^{n-1}(Y) \ar[r] & HX^n(\cY) \ar[r] & \displaystyle\lim_{Y \in \cY} HX^n(Y) \ar[r] & 0.
\end{tikzcd}
\]
In the case that $\cY = \{Y_0\}$ is generated by a single subset, the restriction maps $HX^n(Y') \to HX^n(Y)$ are all isomorphisms for $Y' \supseteq Y \supseteq Y_0$ by coarse invariance of coarse cohomology \cite{WulffCoarseCohomology}, so the $\lim^1$ term vanishes and the restriction map induces isomorphisms
\[
 HX^n(\cY) \cong HX^n(Y)
 \]
 for any $Y \supseteq Y_0$; see \cite[\S1]{HigsonRoeYu}.
 \end{remark}

The support condition ensures that for any given big family $\cY$ in $M$, coarse cochains on $\cY$ are integrable against coarse chains supported on $\cY$, providing a bilinear pairing
\begin{equation}
\label{CoarsePairing}
CX^n(\cY) \times CX_n(\cY) \longrightarrow \C, \qquad \langle\theta, \mu\rangle = \int \theta d \mu.
\end{equation}
By the transformation formula, we have
\begin{equation*}
  \langle \delta \theta, \mu\rangle = \langle \theta, \partial \mu\rangle, 
\end{equation*}
hence this pairing descends to a pairing between coarse homology and coarse cohomology.

It will be important below that  coarse cohomology can also be computed using the anti-symmetrization of the relevant cochain complexes.
Explicitly, for a big family $\cY$ in $M$, we denote by $CX^\bullet_\alpha(\cY) \subset CX^\bullet(\cY)$ the subcomplex consisting of \emph{anti-symmetric} coarse cochains, i.e., those cochains $\theta \in CX^\bullet(\cY)$ satisfying
\[
\theta(x_{\sigma_0}, \dots, x_{\sigma_n}) = \sgn(\sigma) \cdot \theta(x_0, \dots, x_n)
\]
for any permutation $\sigma$.
Particular examples of anti-symmetric $n$-cochains are the functions
\[
f_0 \wedge \cdots \wedge f_n = \sum_{\sigma \in S_{n+1}} \sgn(\sigma) \cdot f_{\sigma_0} \otimes \cdots \otimes f_{\sigma_n}
\]
for bounded Borel functions $f_0, \dots, f_n$ on $M$ (of course, to give an element of $CX_\alpha^n(\cY)$, the wedge product also has to satisfy the relevant support condition).
On these particular cochains, the Alexander-Spanier differential takes the particularly simple form
\begin{equation}
\label{SimpleASDifferential}
\delta(f_0 \wedge \cdots \wedge f_n)  = \mathbf{1} \wedge f_0 \wedge \cdots \wedge f_n.
\end{equation}

By \cite[Prop.~3.26]{Roe-cohom}, the inclusion \begin{equation}
\label{InclusionAnti}
 	CX^\bullet_\alpha(\cY) \to CX^\bullet(\cY)
\end{equation}
 is a quasi-isomorphism, and so the anti-symmetric complex also computes the coarse cohomology groups.
%
%

Dually, we define the chain complex $CX^\alpha_\bullet(\mathcal{Y})$ as the quotient of $CX_\bullet(\mathcal{Y})$ where two measures $\mu_1$, $\mu_2$ are identified if they arise from each other by a signed permutation of the variables.
Equivalently, they are identified if $\mu_1$ and $\mu_2$ have the same pairing with any anti-symmetric coarse cochain.

In the homology case, we do \emph{not} know whether the quotient map
\begin{equation}
\label{MapToAntisymmetricQuotient}
CX_\bullet(\mathcal{Y}) \longrightarrow CX_\bullet^\alpha(\mathcal{Y}) 
\end{equation}
is a quasi-isomorphism; Roe's proof does not generalize to the homology case.
Nevertheless, the \emph{anti-symmetrization} chain map
\begin{equation}
\label{AntiSymmetrization}
\mathrm{A}(\mu) : = \frac{1}{(n+1)!} \sum_{\sigma \in S_{n+1}} \mathrm{sgn}(\sigma) \cdot \sigma_*\mu
\end{equation}
provides a right inverse for the quotient \eqref{MapToAntisymmetricQuotient}, which identifies $CX_\bullet^\alpha(\mathcal{Y})$ with its image in $CX_\bullet(\cY)$. 
This provides a direct sum decomposition 
\begin{equation}
\label{DirectSumCoarseHomology}
CX_\bullet(\mathcal{Y}) \cong CX_\bullet^\alpha(\mathcal{Y}) \oplus \ker(\mathrm{A})_\bullet,
\end{equation}
and, correspondingly, a direct sum decomposition of the homology groups. Throughout, we often identify elements of $CX_\bullet^\alpha(\mathcal{Y})$ with their image in $CX_\bullet(\mathcal{Y})$ under the anti-symmetrization map.

Two measures $\mu_1$ and $\mu_2$ are identified in the quotient $CX_\bullet^\alpha(\mathcal{Y})$ if  and only if they pair identically with any anti-symmetric coarse cochain using \eqref{CoarsePairing}.
This characterization shows that the pairing \eqref{CoarsePairing} descends to a pairing 
\[
CX^n_\alpha(\mathcal{Y}) \times CX_n^\alpha(\mathcal{Y}) \longrightarrow \mathbb{C},
\]
which can be computed by choosing an arbitrary representative of the equivalence class in the quotient complex.
We do not know if the second summand on the right hand side of \eqref{DirectSumCoarseHomology} has trivial homology, but since by Roe's result, any coarse cochain has an anti-symmetric representative, elements of this summand pair trivially with any coarse cochain.

\subsection{The Mayer-Vietoris sequence in coarse homology}
\label{SectionMCCohomology}

In this section, we discuss the Mayer-Vietoris sequence for a decomposition of a proper metric space $M$.
This is rather standard material (see, e.g., \cite{HigsonRoeYu}) but does not appear in the literature in our formulation, using big families and the measure theoretic description of coarse homology, so we recall the relevant material in this language.

\begin{proposition}
\label{PropMayerVietoris}
Let $\cX$, $\cY$ be two big families in $M$.
Then there are short exact sequences of chain complexes
\[
\begin{tikzcd}[column sep=0.5cm, row sep=0.4cm]
0 \ar[r]
&
CX_\bullet(\cX \Cap \cY) \ar[r]
&
CX_\bullet(\cX) \oplus CX_\bullet(\cY)
\ar[r]
&
CX_\bullet(\cX \Cup \cY)
\ar[r]
&
0
\end{tikzcd}
\]
where the first map is the diagonal inclusion, while the second map is given by
\[
(\mu_1, \mu_2) \longmapsto \mu_1 -  \mu_2.
\]
Similarly, we get a short exact sequence when replacing each complex by the quotient complex of anti-symmetric coarse chains.
\end{proposition}

\begin{proof}

Injectivity of the left map is clear.

To see surjectivity of the right map, let $\mu \in CX_n(\cX \Cup \cY)$ be supported on $(X \cup Y)^{n+1} \cap \Delta_R$ for members $X \in \cX$, $Y \in \cY$ and $R \geq 0$.
Then 
\[
\mu_1 := \big(X_{S} \otimes {\underbrace{\mathbf{1} \otimes \cdots \otimes \mathbf{1}}_{n}}\big) \cdot \mu
\] 
is contained in $CX_n(\cX)$ for any $S \geq 0$, as $\mu_1$ vanishes outside
\[
X_S \times M^n \cap \Delta_R \subseteq X_{S+R}^{n+1} \cap \Delta_R.
\] 
The difference $\mu_2 := \mu_1 - \mu$ is given by the same formula, but with $-X_S^c$ replacing $X_S$.
The same calculation shows that $\mu_2$ is supported on $(X_S^c)_R \cap (X \cup Y)$, which is contained in $Y$ if $S$ is large enough. 
Hence for such a choice of $S$, we have $\mu_2 \in CX_n(\cY)$.
This shows that $\mu = \mu_1 - \mu_2$ is contained in the image of the right map.

%
%

To see exactness in the middle, we observe that since the second map is $(\mu_1, \mu_2) \mapsto \mu_1 - \mu_2$, elements in the kernel must satisfy $\mu_1 = \mu_2 =:\mu $.
Exactness in the middle therefore follows from the obvious identity
\[
  CX_n(\cX) \cap CX_n(\cY) = CX_n(\cX \Cap \cY). 
\]
%
%
The argument above also works for the  anti-symmetric quotient complexes. 
We only need to modify $\mu_1$ to 
\[
\mu_1 = \frac{1}{n+1}\sum_{i=0}^n  \,({\underbrace{\mathbf{1} \otimes \cdots \otimes \mathbf{1}}_i} \otimes X_S \otimes {\underbrace{\mathbf{1} \otimes \cdots \otimes \mathbf{1}}_{n-i}}) \cdot \mu.
\]
Then if $\mu'$ is obtained from $\mu$ by a signed permutation, then $\mu_1'$ and $\mu_1$ define the same element in the quotient $CX_n^\alpha(\cY)$. 
So the construction is well-defined on the quotient.
\end{proof}

By the general principle of  homological algebra, the above proposition provides, for any pair $\cX$, $\cY$ of big families in $M$, a long exact sequence in homology,
\[
\hspace{-0.1cm}
\begin{tikzcd}[column sep=0.7cm]       
\cdots \arrow{r}  
&[0.3cm]
HX_n(\cX\Cap \cY) \ar[r] 
\arrow[phantom, "\partial_{\mathrm{MV}}"{coordinate, name=Z}]{d}
&
HX_n(\cX) \oplus HX_n(\cY) \ar[r]  &
HX_n(\cX \Cup \cY) 
\ar[d, near start, shift left=23, phantom, "\partial_{\mathrm{MV}}"]
  \arrow[
    rounded corners,
    to path={
      -- ([xshift=2ex]\tikztostart.east)
      |- (Z) [near end]\tikztonodes
      -| ([xshift=-2ex]\tikztotarget.west)
      -- (\tikztotarget)
    }
  ]{dll} 
  \\
 &
 HX_{n-1}(\cX \Cap \cY) \arrow{r} & HX_{n-1}(\cX) \oplus HX_{n-1}(\cY) \ar[r] &  HX_{n-1}(\cX \Cup \cY) \ar[r] &
\cdots &
\end{tikzcd}
\]

\begin{remark}
\label{RemarkAlternativeMVboundary}
Below, we will use the following description of the Mayer-Vietoris boundary map $\partial_{\mathrm{MV}}$, using a version of the excision isomorphism in coarse homology.
Consider the commutative diagram 
\[
\begin{tikzcd}[column sep=0.5cm, row sep=0.7cm]
0 \ar[r]
&
CX_\bullet(\cX \Cap \cY) \ar[r]
\ar[dd, equal]
&
CX_\bullet(\cX) \oplus CX_\bullet(\cY)
\ar[r]
\ar[dd]
&
CX_\bullet(\cX \Cup \cY)
\ar[r]
\ar[d]
&
0
\\
&&&
\displaystyle\frac{CX_\bullet(\cX\Cup \cY)}{CX_\bullet(\cY)}
\\
0 \ar[r]
&
CX_\bullet(\cX \Cap \cY) \ar[r]
&
CX_\bullet(\cX)
\ar[r]
&
\displaystyle\frac{CX_\bullet(\cX)}{CX_\bullet(\cX \Cap \cY)}
\ar[r]
\ar[u, "\cong"]
&
0,
\end{tikzcd}
\]
where the bottom right vertical map is the excision isomorphism \eqref{QuotientMap}.
%
%
By naturality of boundary maps, we see that the Mayer-Vietoris boundary map may be written as the composition
\[
\begin{tikzcd}[column sep=0.5cm]
\partial_{\mathrm{MV}} : HX_n(\cX \Cup\cY) \ar[r]
&
HX_n(\cX \Cup \cY, \cY) \cong HX_n(\cX, \cX \Cap \cY)
\ar[r, "\partial"]
&
HX_{n-1}(\cX \Cap \cY),
\end{tikzcd}
\]
where the first map is the quotient map, in the middle, we used the excision isomorphism from Prop.~\ref{PropExcision} and the right map is the boundary map for the bottom short exact sequence.
In the above chain level description of $\partial_{\mathrm{MV}}$, we may also replace each occurrence of $CX_\bullet$ by the anti-symmetric version $CX_\bullet^\alpha$.
\end{remark}

\subsection{The Mayer-Vietoris sequence in coarse cohomology}

The Mayer-Vietoris sequence in coarse cohomology is more subtle. 
We start with the following lemma.

\begin{lemma}
\label{LemmaCohHalfMV}
Let $\cX$ and $\cY$ be two big families in a proper metric space $M$.
Then the sequence of cochain complexes
\begin{equation}
\label{SESCohomology}
\begin{tikzcd}[column sep=0.8cm]
0 \ar[r]
&
CX^\bullet(\cX \Cup \cY) \ar[r]
&
CX^\bullet(\cX) \oplus CX^\bullet( \cY)
\ar[r]
&
CX^\bullet(\cX \Cap \cY)
\end{tikzcd}
\end{equation}
is exact, where the maps are dual to those of the homological Mayer-Vietoris sequence.
\end{lemma}

%

\begin{proof}
Injectivity of the left map is clear, as $CX^\bullet(\cX \Cup \cY)$ is contained in both $CX^\bullet(\cX)$ and $CX^\bullet(\cY)$.

Exactness in the middle follows from the equality
\[
  CX^n(\cX) \cap CX^n(\cY) = CX^n(\cX \Cup \cY).
\]
for general big families $\cX$ and $\cY$.
The inclusion $(\supseteq)$ is clear.
To see the inclusion $(\subseteq)$, we observe that for $X \in \cX$, $Y \in \cY$ and any $R \geq 0$, we have
\[
\Delta_R \cap (X \cup Y)^{n+1} \subseteq \Delta_R \cap \big( X_{2R}^{n+1} \cup Y_{2R}^{n+1}) 
\]
Therefore, a cochain that is both coarse on $X$ and coarse on $Y$ is coarse on $X \cup Y$.
\end{proof}

In contrast to the case of homology, we do not know whether the right map in \eqref{SESCohomology} is surjective in general. 
However, we have the following lemma.

\begin{lemma}
\label{LemmaLocalization}
Let $\cX$, $\cY$ be big families in $M$ and let $X$ be a Borel member of $\cX$. 
Then given any  $\theta \in CX^n(\cX \Cap \cY)$, the cochain
\[
\tilde{\theta}^c = (X \otimes {\underbrace{\mathbf{1} \otimes \cdots \otimes \mathbf{1}}_n})\cdot \theta
\]
is contained in $CX^n(\cY)$.
\end{lemma}

\begin{proof}
Let $Y$ be a member of $\cY$ and let $R \geq 0$.
Suppose that $x = (x_0, \dots, x_n)$ is contained in $\supp(\tilde{\theta}^c)\cap \Delta_R \cap Y^{n+1}$.
Then $x_0 \in X$, so from $x \in \Delta_R$, it follows that $x_i \in X_{2R}$ for each $i =1, \dots, n$.
So
\[
\supp(\tilde{\theta}^c)\cap \Delta_R \cap Y^{n+1} \subseteq \supp({\theta})\cap \Delta_R \cap (Y\cap X_{2R})^{n+1},
\]
which is bounded by the assumption that $\theta \in CX^n(\cX \Cap \cY)$.
This verifies the support condition for $\tilde{\theta}^c$ to be a member of $CX^n(\cY)$.
\end{proof}

\vspace{3mm}

\begin{proposition}
\label{PropMVXXc}
Let $\cZ$ be a big family in $M$ and let $X \subseteq M$ be a Borel subset.
Then we have short exact sequences of cochain complexes
\[
\begin{tikzcd}[column sep=0.5cm, row sep=0.4cm]
0 \ar[r]
&
CX^\bullet(\cZ) \ar[r]
&
CX^\bullet(\{X\}\Cap\cZ ) \oplus CX^\bullet( \{X^c\}\Cap\cZ)
\ar[r]
&
CX^\bullet(\partial X \Cap \cZ)
\ar[r]
&
0.
\end{tikzcd}
\] 
\end{proposition}

\begin{proof}
By Lemma~\ref{LemmaCohHalfMV}, applied with $\cX = \{X\} \Cap \cZ$ and $\cY = \{X^c\} \Cap \cZ$, it suffices to show that the right map is surjective.
In other words, we need to show that any cochain $\theta$ in $CX^n(\partial X \Cap \cZ)$ can be written as $\theta = \tilde{\theta} + \tilde{\theta}^c$ with $\tilde{\theta} \in CX^n(\{X\} \Cap \cZ)$ and $\tilde{\theta}^c \in CX^n(\{X^c\} \Cap \cZ)$. 
By Lemma~\ref{LemmaLocalization}, the choices
\begin{equation*}
\tilde{\theta} = (X^c \otimes \mathbf{1} \otimes \cdots \otimes \mathbf{1} ) \cdot \theta, \qquad \text{and} \qquad \tilde{\theta}^c = (X \otimes \mathbf{1} \otimes \cdots \otimes \mathbf{1} ) \cdot \theta
\end{equation*}
achieve this.
%
%
\end{proof}

The above proposition yields a long exact sequence
\[
\begin{tikzcd}[column sep=0.5cm]       
\cdots \arrow{r}  
&[0.3cm]
HX^n(\cZ)
\ar[d, phantom, shift right=20, "\delta_{\mathrm{MV}}", near end]
\ar[r]
\arrow[phantom, ""{coordinate, name=Z}]{d}
&
HX^n(\{X\} \Cap \cZ) \oplus HX^n(\{X^c\} \Cap \cZ) \ar[r]  &[-0.1cm]
HX^n(\partial X \Cap \cZ) 
  \arrow[
    rounded corners,
    to path={
      -- ([xshift=2ex]\tikztostart.east)
      |- (Z) [near end]\tikztonodes
      -| ([xshift=-2ex]\tikztotarget.west)
      -- (\tikztotarget)
    }
  ]{dll} 
  \\
 &
 HX^{n+1}(\cZ) \arrow{r} & HX^{n+1}(\{X\} \Cap \cZ) \oplus HX^{n+1}(\{X^c\} \Cap \cZ) \ar[r] &  
\cdots \qquad\qquad &
\end{tikzcd}
\]
in coarse cohomology.
Moreover, the sequence is dual to the homological Mayer-Vietoris sequence, in the sense that we have the duality
\begin{equation}
\label{DualityMVdifferentials}
  \langle \delta_{\mathrm{MV}}\theta, \mu\rangle = \langle \theta, \partial_{\mathrm{MV}} \mu\rangle
\end{equation}
of the Mayer-Vietoris differentials.

\begin{remark}
\label{RemarkAntiSymmetricVersion}
If $\cZ$ is a big family and $X$ is a Borel subset, we may give an explicit formula for the Mayer-Vietoris boundary 
\[
\delta_{\mathrm{MV}} : HX^n(\partial X \Cap \cZ) \longrightarrow HX^{n+1}(\cZ),
\]
as this is the boundary map for the long exact sequence in cohomology obtained from the short exact sequence of cochain complexes 
\[
\begin{tikzcd}[column sep=0.5cm, row sep=0.4cm]
0 \ar[r]
&
CX^\bullet_\alpha(\cZ) \ar[r]
&
CX^\bullet_\alpha(\{X\}\Cap\cZ ) \oplus CX^\bullet_\alpha( \{X^c\}\Cap\cZ)
\ar[r]
&
CX^\bullet_\alpha(\partial X \Cap \cZ)
\ar[r]
&
0.
\end{tikzcd}
\] 
Namely, for any $\theta \in CX^n_\alpha(\partial X \Cap \cZ)$, we may write $\theta = \tilde{\theta} + \tilde{\theta}^c$, where $\tilde{\theta}^c \in CX^n_\alpha(\{X^c\} \Cap \cZ)$ is defined by the formula 
\begin{equation}
\label{AntisymmetricLocalization}
\tilde{\theta}^c := \frac{1}{n+1}\sum_{i=0}^n  \,({\underbrace{\mathbf{1} \otimes \cdots \otimes \mathbf{1}}_i} \otimes X \otimes {\underbrace{\mathbf{1} \otimes \cdots \otimes \mathbf{1}}_{n-i}}) \cdot \theta,
\end{equation}
 and $\tilde{\theta} \in CX^n_\alpha(\{X\} \Cap \cZ)$ is defined by the same expression but with $X$ replaced by $X^c$.
 It follows from Lemma~\ref{LemmaLocalization} that these cochains have the desired support properties.
By general principles of homological algebra, the boundary map is therefore given by the formula
\begin{equation*}
\delta_{\mathrm{MV}}([\theta]) = [\delta(\tilde{\theta})] = -[\delta(\tilde{\theta}^c)], 
\end{equation*}
where on the right hand side, $\delta$ is the Alexander-Spanier differential.
From the exactness of the sequence of cochain complexes, it follows that $\delta(\tilde{\theta})$ and $-\delta(\tilde{\theta}^c)$  agree and are actually contained in $CX^{n+1}_\alpha(\cZ)$.
\end{remark}

\section{$K$-theory, cyclic homology and the Chern character}
\label{Section4}

In this section, we give a recollection of the tools we use for the proof of our main integrality results: $K$-theory, cyclic homology, the Chern character and coarse character maps.

\subsection{Cyclic cohomology and the coarse character map}
\label{SectionCyclicHomology}

A standard textbook reference on cyclic homology is Loday's book \cite{Loday}.
The coarse character map to coarse homology was defined first by Roe \cite{Roe-cohom} as a map in the dual direction (sending coarse \emph{co}homology to cyclic \emph{co}homology). 
As we are interested in associating coarse homology classes to general operators from $\sB(M)$, not just those coming from a Dirac index (which Roe was focussing on), the homological definition will be more convenient and natural; compare also \cite[\S3.3]{Engel}, where this was considered before.

Let $\sA$ be an algebra over $\C$.

\paragraph{Cyclic homology.}

We write $C_n(\sA) = \sA^{\otimes(n+1)}$, which is a chain complex with the usual Hochschild differential
\[
b(a_0 \otimes \cdots \otimes a_n) = \sum_{i=0}^{n-1}(-1)^i a_0 \otimes \cdots \otimes a_i a_{i+1} \otimes \cdots \otimes a_n + (-1)^n a_na_0 \otimes a_1 \otimes \cdots \otimes a_{n-1}.
\]
Let $C_\bullet^\lambda (\sA) := C_\bullet(\sA)/\mathrm{im} (1-\lambda)$ be the quotient complex of \emph{cyclic chains}, in which two chains are identified if they are a (signed) cyclic permutation of each other.
Here the cyclic permutation operator $\lambda: C_n(\sA) \to C_n(\sA)$ is given by
\begin{equation*}
	\lambda(a_0 \otimes \cdots \otimes a_n) = (-1)^n a_n \otimes a_0 \otimes \cdots \otimes a_{n-1}.
\end{equation*} 
The usual Hochschild differential of $C_\bullet(\sA)$ descends to $C^\lambda_\bullet(\sA)$, turning it into a chain complex, and its homology is by definition the \emph{cyclic homology} $HC_\bullet(\sA)$ of $\sA$.
For example, writing $[\sA,\sA]$ for the subspace generated by commutators of elements in $\sA$, we have
\begin{equation}\label{ZerothCyclicHomology}
HC_0(\sA)=C^\lambda_0(\sA)/b(C^\lambda_1(\sA))=\sA/[\sA,\sA].
\end{equation}

If $\sJ$ is an ideal in $\sA$, then the \emph{relative cyclic homology} groups $HC_n(\sA, \sJ)$ are defined as the homology of the kernel complex
\[
C^\lambda_\bullet(\sA, \sJ) := \ker\big(C^\lambda_\bullet(\sA) \longrightarrow C^\lambda_\bullet(\sA/\sJ)\big).
\]
From the corresponding short exact sequence of chain complexes, one obtains a boundary map 
\[
\partial : HC_n(\sA/\sJ) \longrightarrow HC_{n-1}(\sA, \sJ)
\]
in cyclic homology.

\begin{remark}
There is an obvious chain map $C^\lambda_\bullet(\sJ) \to C^\lambda_\bullet(\sA, \sJ)$, and one says that $\sJ$ \emph{satisfies excision in cyclic homology} if this map is a quasi-isomorphism. 
By a famous result of Wodzicki, this is the case if and only if $\sJ$ is $H$-unital \cite{wodzicki1989excision}.
The localization ideals $\sB(\cY)$ are likely \emph{not} $H$-unital in the case that the $M$-module is locally infinite-dimensional, so we do \emph{not} expect excision to hold.
\end{remark}

A key feature of the theory is the existence of \emph{periodicity operators}
\begin{equation}
\label{Soperator}
S : HC_{n+2}(\sA) \longrightarrow HC_n(\sA).	
\end{equation}
While these are slightly awkward to describe in terms of the cyclic complex $C^\lambda_\bullet(\sA)$ (see, however \cite[Thm.~2.2.7]{Loday}), we will not need an explicit formula for these.

\paragraph{The coarse character map.}

 The cyclic homology of the localization ideals $\sA=\sB(\cY)$ for big families $\cY$ in a proper metric space $M$ is closely related to coarse homology via the \emph{coarse character map} 
\begin{equation}
\label{CharacterMap}
\chi_* : HC_n(\sB(\cY)) \longrightarrow HX_n(\cY),
\end{equation}
see \cite[\S4.2]{Roe-cohom}.
This map is induced by a chain map on the underlying complexes, which we will describe now.
(In fact, the map described here is the dual of the character map introduced in \cite[\S4.2]{Roe-cohom}.)

The map $\chi$ is defined as follows.
Given elements $T_0, \dots, T_n \in \sB(\cY)$, 
let us consider the multilinear functional $\chi(T_0 \otimes \dots \otimes T_n)$ on the space $C_c(M)$ of compactly supported continuous functions on $M$ defined by
\begin{equation}
\label{DefCharacterMap}
  \chi(T_0\otimes \dots \otimes T_n)(f_0, \dots, f_n) := \Tr(f_0 T_0 \cdots f_n T_n).
\end{equation}
Note that because the $T_i$ are locally trace-class and the $f_i$ are compactly supported, each of the terms $f_i T_i$ is trace-class, and so is their product;
hence the expression is well-defined.
It is shown in \cite[\S4.2]{Roe-cohom} that \eqref{DefCharacterMap} is locally bounded as a linear functional on the subspace of decomposables in $C_c(M^{n+1})$, therefore extends by continuity to a linear functional on $C_c(M^{n+1})$ and hence is given by integration with respect to some locally finite complex-valued Borel measure on $M^{n+1}$ by the Kakutani--Markov--Riesz theorem.
Because each $T_i$ has finite propagation, this measure
has support in some $R$-thickening $\Delta_R$ of the multi-diagonal, hence we obtain a well-defined element $\chi(T_0 \otimes \dots \otimes T_n)$ in $CX_n(M)$.

 Moreover, if each $T_i$ is supported on some member $Y_i$ of $\cY$, the measure $\chi(T_0 \otimes \dots \otimes T_n)$ is supported on $Y_0 \cup \cdots \cup Y_n$.
 This gives an element $\chi(T_0 \otimes \dots \otimes T_n)$ in $CX_n(\cY)$.
 This defines a map 
 \[
 \chi : C_n(\sB(\cY)) \longrightarrow CX_n(\cY)
 \]
 for each big family $\cY$ in $M$.

\begin{lemma}
The map $\chi$ defined above descends to a chain map 
\begin{equation*}
  \chi : C_\bullet^\lambda(\sB(\cY)) \longrightarrow CX^\alpha_\bullet(\cY). 
\end{equation*}
\end{lemma}

Viewing the complex of anti-symmetric chains as a subcomplex of all chains using the anti-symmetrization map \eqref{AntiSymmetrization}, this yields the desired character map \eqref{CharacterMap} after passing to homology. 

\begin{proof}
We first show that $\chi$ indeed descends to the quotient $C^\lambda_n(\sB(\cY))$. 
For any decomposable cochain $f_0\otimes\dots\otimes f_n\in C_c(M^{n+1})$ which is coarse on $\cY$, we use cyclicity of the trace to deduce that
\begin{align*}
\chi\big(\lambda(T_0 \otimes \cdots \otimes T_n)\big)(f_0, \dots, f_n)
&= (-1)^n \cdot \chi(T_n \otimes T_0 \otimes \cdots \otimes T_{n-1})(f_0, \dots, f_n)
\\
&= (-1)^n \cdot \Tr(f_0 T_n f_1 T_0 f_2 T_1 \cdots f_n T_{n-1})
\\
&= (-1)^n \cdot \Tr(f_1 T_0 f_2 T_1 \cdots f_n T_{n-1}f_0 T_n )
\\
&= (-1)^n \cdot \chi(T_0 \otimes \cdots \otimes T_n)(f_1, \ldots, f_n, f_0).
\end{align*}
This shows that $\chi(\lambda(T_0 \otimes \cdots \otimes T_n))$ has the same integration pairing with anti-symmetric coarse cochains on $\cY$ as $\chi(T_0 \otimes \cdots \otimes T_n)$ does. 
Hence both measures agree in the quotient $CX^\alpha_n(\cY)$.

To see that $\chi$ is a chain map, we calculate
\begin{align*}
\chi\big(b(T_0 \otimes \cdots \otimes T_n)\big)&(f_0, \dots, f_{n-1}) 
\\
&= \sum_{i=0}^{n-1} (-1)^i \chi\big(T_0 \otimes \cdots \otimes T_{i}T_{i+1} \otimes \cdots \otimes T_n\big)(f_0, \dots, f_{n-1})	
\\
&\qquad + (-1)^n \chi\big(T_n T_0 \otimes \cdots  \otimes T_{n-1}\big)(f_0, \dots, f_{n-1})
\\
&= \sum_{i=0}^{n-1} (-1)^i \chi\big(T_0 \otimes \cdots  \otimes T_n\big)(f_0, \dots, \underbrace{\mathbf{1}}_{i+1}, \dots f_{n-1})	
\\
&\qquad + (-1)^n \chi\big(T_n \otimes T_0 \otimes \cdots  \otimes T_{n-1}\big)(f_0, \mathbf{1}, f_1, \dots, f_{n-1})
\\
&= - \chi(T_0 \otimes \cdots \otimes T_n) \big(\delta(f_0 \otimes \dots \otimes f_{n-1})\big)
\\
&\qquad+ \chi(T_0 \otimes \cdots \otimes T_n)(\mathbf{1}, f_0, \dots, f_{n-1})\\
&\qquad + (-1)^n \chi\big(T_0 \otimes \cdots  \otimes T_{n}\big)(\mathbf{1}, f_1, \dots, f_{n-1},f_0).
\end{align*}
Thus, we have
\[
\chi\big(b(T_0 \otimes \cdots \otimes T_n)\big)(f_0\wedge \dots \wedge f_{n-1}) = - \chi(T_0 \otimes \cdots \otimes T_n)\big(\delta(f_0\wedge \dots \wedge f_{n-1})\big),
\]
and it follows that $\chi(b(T_0 \otimes \cdots \otimes T_n))$ has  the same pairing with anti-symmetric cochains as $-\partial \chi(T_0 \otimes \cdots \otimes T_n)$, in other words, they agree in the quotient $CX^\alpha_{n-1}(\cY)$.
\end{proof}

If $\cX$ and $\cY$ are two big families in $M$, we also obtain a map
\begin{align*}
\chi : C^\lambda_\bullet\left(\frac{\sB(\cX)}{\sB(\cX \Cap \cY)}\right) &\longrightarrow CX_\bullet^\alpha(\cX, \cX \Cap \cY) = \frac{CX_\bullet^\alpha(\cX)}{CX_\bullet^\alpha(\cX \Cap \cY)},
\end{align*}
whose well-definedness follows from the fact that the measure $\chi(T_0 \otimes \cdots \otimes T_n)$ is supported on $\cX \Cap \cY$ as soon as \emph{one} of the factors $T_i$ is supported on $\cX \Cap \cY$.
This implies that $\chi:C^\lambda_\bullet(\sB(\cX)) \to CX^\alpha_\bullet(\cX)$ restricts to a map
\[
\chi : C^\lambda_\bullet\big(\sB(\cX), \sB(\cX \Cap \cY)\big) \longrightarrow CX^\alpha_\bullet(\cX \Cap \cY)
\]
 between the kernel complexes.

\begin{lemma}
\label{CharacterBoundary}
Let $\cX$ and $\cY$ be two big families in $M$. 
Then for each $n \in \Z$, the following diagram commutes.
\begin{equation*}
\begin{tikzcd}
\displaystyle HC_n\left(\frac{\sB(\cX)}{\sB(\cX \Cap \cY)}\right)
\ar[r, "\chi_*"]
\ar[d, "\partial"']
&
HX_n(\cX, \cX \Cap \cY)
\ar[d, "\partial"']
\\
HC_{n-1}\big(\sB(\cX), \sB(\cX \Cap \cY)\big)
\ar[r, "\chi_*"]
&
HX_{n-1}(\cX \Cap \cY).
\end{tikzcd}
\end{equation*}
\end{lemma}

\begin{proof}
We have the  morphism 
\[
\begin{tikzcd}[column sep=0.8cm]
0 \ar[r]
&
C^\lambda_\bullet\big(\sB(\cX), \sB(\cX \Cap \cY)\big)
\ar[r]
\ar[d, "\chi_*"']
& 
C^\lambda_\bullet(\sB(\cX))
\ar[r]
\ar[d, "\chi_*"']
&
\displaystyle C^\lambda_\bullet\left(\frac{\sB(\cX)}{\sB(\cX \Cap \cY)}\right)
\ar[d, "\chi_*"']
\ar[r]
&
0
\\
0 \ar[r]
&
CX_\bullet^\alpha(\cX \Cap \cY)
\ar[r]
&
CX_\bullet^\alpha(\cX)
\ar[r]
&
\displaystyle \frac{CX_\bullet^\alpha(\cX)}{CX_\bullet^\alpha(\cX \Cap \cY)}
\ar[r]
&
0
\end{tikzcd}
\]
of short exact sequences of chain complexes. 
Taking homology, this induces the claimed commutative diagram involving the induced boundary maps.
\end{proof}

It will be convenient below that we may trivially extend the coarse character map to the unitization $\sB(\cY)^+$ by setting
\[
\widetilde{\chi}\big((T_0 + \lambda_0) \otimes  \dots \otimes ( T_n + \lambda_n)\big) := \chi(T_0 \otimes \cdots \otimes T_n).
\]
It is a standard fact that this again yields a chain map and hence a coarse character map
\[
\widetilde{\chi}_* : HC_n(\sB(\cY)^+) \longrightarrow HX_n(\cY).
\]

\subsection{$K$-theory and the Mayer-Vietoris sequence}

Throughout this section, let $\sA$ be a (not necessarily unital) algebra over $\C$.

\paragraph{Algebraic $K$-theory.}

Recall that for each $n \in \Z$, there is an abelian group $K_n^{\mathrm{alg}}(\sA)$, called the $n$-th \emph{algebraic $K$-theory} group of $\sA$.
For $n\in\{0, 1\}$, these groups have the following explicit descriptions:
\begin{enumerate}
\item[(0)]
For unital $\sA$, the algebraic $K$-theory group $K_0^{\mathrm{alg}}(\sA)$ can be described as the Grothendieck group of the set of equivalence classes of idempotents in matrix algebras over $\sA$, where two idempotents are identified if they are conjugate.
For non-unital algebras, one defines $K_0^{\mathrm{alg}}(\sA) := \ker(K_0^{\mathrm{alg}}(\sA^+) \to K_0^{\mathrm{alg}}(\C))$, where $\sA^+$ is the unitization of $\sA$.
\item[(1)]
For unital $\sA$, one defines $K_1^{\mathrm{alg}}(\sA)$ as the abelianization of the group $\GL_\infty(\sA)$ consisting of invertible infinite matrices with $\sA$-valued entries, which differ from the identity matrix only in finitely many places. 
For non-unital $\sA$ one defines $K_1^{\mathrm{alg}}(\sA)$ as in the $n=0$ case as a subgroup of $K_1^{\mathrm{alg}}(\sA^+)$. 
We will use that any element in $K_1^{\mathrm{alg}}(\sA)$ can be represented by an invertible element $U \in M_k(\sA^+)$ such that $U-1 \in M_k(\sA)$.
We denote the groups of such elements by $\GL_k(\sA)^+$, and set $\GL_k(\sA)^+ := \GL_k(\sA)$ if $\sA$ is already unital.
\end{enumerate}

We will not need explicit descriptions of the algebraic $K$-groups in higher or lower degrees, but will only use the formal properties of the algebraic $K$-functor.
Most importantly, whenever $\sJ \subseteq \sA$ is an ideal, there are \emph{relative} $K$-theory groups $K_n^{\mathrm{alg}}(\sA, \sJ)$, $n \in \Z$, which fit into a long exact sequence
\[
\cdots
\longrightarrow
K_n^{\mathrm{alg}}(\sA, \sJ) 
\longrightarrow
K_n^{\mathrm{alg}}(\sA)
\longrightarrow
K_n^{\mathrm{alg}}(\sA/\sJ)
\longrightarrow
K_{n-1}^{\mathrm{alg}}(\sA, \sJ)
\longrightarrow
\cdots
\]
One has natural maps $K_n^{\mathrm{alg}}(\sJ) \to K_n^{\mathrm{alg}}(\sA, \sJ)$ and $\sJ$ is said to \emph{satisfy excision in algebraic $K$-theory} if this map is an isomorphism for every algebra $\sA$ containing $\sJ$ as an ideal.

\paragraph{Topological $K$-theory.}

If $\sA$ is a Banach algebra, one may also define its \emph{topological $K$-theory groups} $K_n^{\mathrm{top}}(\sA)$, $n \in \Z$.
One then has functorial comparison maps
\begin{equation}
\label{ComparisonKalgKtop}
K_n^{\mathrm{alg}}(\sA) \longrightarrow K_n^{\mathrm{top}}(\sA)
\end{equation}
from algebraic to topological $K$-theory.
This comparison map is always an isomorphism in degree zero.
The comparison map is an isomorphism in \emph{all} degrees if $\sA$ is, for example, a stable $C^*$-algebra by the resolution of Karoubi's conjecture due to Suslin--Wodzicki \cite[\S10]{SuslinWodzicki1992} (see also \cite{CortinasThom}), but in general (for example in the case of $\sA = \C$), it is far from being an isomorphism.
In topological $K$-theory, there is unconditional excision for \emph{closed} ideals $\sJ \subseteq \sA$, meaning that there is an associated long exact sequence 
\[
\cdots 
\longrightarrow
K_n^{\mathrm{top}}(\sJ) 
\longrightarrow
K_n^{\mathrm{top}}(\sA)
\longrightarrow
K_n^{\mathrm{top}}(\sA/\sJ)
\longrightarrow
K_{n-1}^{\mathrm{top}}(\sJ)
\longrightarrow
\cdots
\]
to every ideal inclusion, which is compatible with the comparison maps \eqref{ComparisonKalgKtop}.
%
%
In contrast to algebraic $K$-theory, the groups in degrees other than $n \in \{0, 1\}$ may be related to the low degree ones using the \emph{Bott periodicity isomorphism} 
\[
\beta: K_n^{\mathrm{top}}(\sA) \stackrel{\cong}{\longrightarrow} K_{n+2}^{\mathrm{top}}(\sA)
\]
making the complex topological $K$-groups 2-periodic, which is natural in $\sA$ and commutes with the boundary maps of the long exact sequences in topological $K$-theory.

\paragraph{Homotopy $K$-theory.}

The algebras $\sB(\cY)$ for big families $\cY$ in a metric space $M$ likely do \emph{not} satisfy excision in algebraic $K$-theory, and they are also not Banach algebras so that the topological $K$-theory functor cannot be applied.
To tackle these issues, we will use a closely related theory, namely \emph{Weibel's homotopy $K$-theory}, which has better formal properties than algebraic $K$-theory \cite{Weibel}.
It associates abelian groups $KH_n(\sA)$, $n \in \Z$, to arbitrary algebras $\sA$ and is designed to be \emph{homotopy invariant}, in the sense that the evaluation maps $\sA[x] \to \sA$ provide isomorphisms $KH_n(\sA[x]) \cong KH_n(\sA)$.
Homotopy $K$-theory comes with functorial comparison maps
\begin{equation}
\label{ComparisonMapKalgKH}
K_n^{\mathrm{alg}}(\sA) \longrightarrow KH_n(\sA)
\end{equation}
and---in contrast to algebraic $K$-theory---satisfies unconditional excision in the sense that we have long exact sequences
\[
\cdots \longrightarrow
KH_n(\sJ) 
\longrightarrow
KH_n(\sA)
\longrightarrow
KH_n(\sA/\sJ)
\longrightarrow
KH_{n-1}(\sJ)
\longrightarrow
\cdots
\]
for \emph{any} ideal $\sJ\subseteq \sA$ \cite[Thm.~2.1]{Weibel}.
Beware, however, that the comparison map \eqref{ComparisonMapKalgKH} is generally \emph{not} an isomorphism, even in degree zero.
What makes $KH$-theory useful for our purposes, however, is its interplay with topological $K$-theory.
Namely, for a Banach algebra $\sA$, the comparison map \eqref{ComparisonKalgKtop} factors naturally through $KH$,
\begin{equation}
\label{FactorizationComparisonMap}
K_n^\mathrm{alg}(\sA) \longrightarrow KH_n(\sA) \longrightarrow K_n^{\mathrm{top}}(\sA)
\end{equation}
where the first map is \eqref{ComparisonMapKalgKH} and the composition is \eqref{ComparisonKalgKtop}.
These comparison maps are compatible with the boundary maps for the corresponding long exact sequences.

\paragraph{The Mayer-Vietoris sequence in $KH$.}

Let $M$ be a proper metric space and let $\cH$ be an $M$-module to form the localization ideals for  big families in $M$.
Let $\cX, \cY$ be two big families in $M$.
Then there is a long exact Mayer-Vietoris sequence in homotopy $K$-theory, 
\[
\hspace{-0.1cm}
\begin{tikzcd}[column sep=0.3cm]       
\cdots \arrow{r}  
&[0.3cm]
KH_n(\sB(\cX\Cap \cY)) \ar[r] 
\arrow[phantom, ""{coordinate, name=Z}]{d}
&
KH_n(\sB(\cX)) \oplus KH_n(\sB(\cY)) \ar[r]  &
KH_n(\sB(\cX \Cup \cY)) 
  \arrow[
    rounded corners,
    to path={
      -- ([xshift=2ex]\tikztostart.east)
      |- (Z) [near end]\tikztonodes
      -| ([xshift=-2ex]\tikztotarget.west)
      -- (\tikztotarget)
    }
  ]{dll} 
  \\
 &
KH_{n-1}(\sB(\cX \Cap \cY)) \arrow{r} & KH_{n-1}(\sB(\cX)) \oplus KH_{n-1}(\sB(\cY)) \ar[r] &
\cdots \hspace{2cm}&
\end{tikzcd}
\]
constructed as follows.
Consider the short exact sequence
\begin{equation}\label{eqn:KHSES}
0 
\longrightarrow
\sB(\cX\Cap \cY) 
\longrightarrow
 \sB(\cX)
\longrightarrow
\displaystyle\frac{\sB(\cX)}{\sB(\cX \Cap \cY)}
\longrightarrow
 0
\end{equation}
of algebras, which yields a long exact sequence in homotopy $K$-theory.
The Mayer-Vietoris boundary map is then the boundary map of this long exact sequence, precomposed with the map induced in $KH$ by the algebra homomorphism
\[
\sB(\cX \Cup \cY) \longrightarrow \frac{\sB(\cX \Cup \cY)}{\sB(\cY)} \cong \frac{\sB(\cX)}{\sB(\cX \Cap \cY)}.
\]

\subsection{Chern characters and Mayer-Vietoris maps}

Quite generally, Chern characters are natural transformations comparing some flavor of $K$-theory with some flavor of cyclic homology.
In this section, we review the general theory needed for the formulation of the Kubo and Kitaev pairings in these terms in Section \ref{SectionIntegralityKubo}. 

\paragraph{Abstract properties.}

For any $n \in \Z$ and any $m \geq 0$, there is a \emph{Chern character} homomorphism
\[
\ch_n^{(m)} : K_n^{\mathrm{alg}}(\sA) \longrightarrow HC_{n+2m}(\sA),
\]
where we sometimes leave out the indices $n, m$, depending on the context.
For fixed $n$, the Chern characters have the property that the different Chern characters are related by the $S$-operator \eqref{Soperator} through the formula 
\begin{equation}
\label{CompatibilityS}
S \ch_n^{(m)} = \ch_n^{(m-1)}.
\end{equation}
This relation is a shadow of the fact that the Chern character actually takes values in \emph{periodic} cyclic homology $HP_n(\sA)$, see \cite[11.4.3]{Loday}. 
Because the latter theory is homotopy invariant (in the sense that the evaluation maps $\sA[x] \to \sA$ induce isomorphisms in $HP$, see \cite{Goodwillie}), the Chern character factorizes naturally through homotopy $K$-theory,
\begin{equation*}
\ch_n^{(m)} : K^{\mathrm{alg}}_n(\sA)
\longrightarrow
KH_n(\sA)
\xrightarrow{~\ch_n~}
HP_n(\sA) 
\longrightarrow
HC_{n+2m}(\sA).
\end{equation*}
The Chern character intertwines the boundary maps corresponding to ideal inclusions $\sJ \subseteq \sA$, leading to commutative diagrams \cite[11.4.8]{Loday}
\begin{equation}
\label{ChernVsBoundary}
\begin{tikzcd}
K_n^{\mathrm{alg}}(\sA/\sJ)
\ar[r] 
\ar[rr, bend left=20, "\ch_n^{(m)}"]
\ar[d, "\partial"']
&
KH_n(\sA/\sJ)
\ar[r, "\ch_n^{(m)}"]
\ar[dd, "\partial"']
&[0.8cm]
HC_{n+2m}(\sA/\sJ)
\ar[d, "\partial"']
\\
K_{n-1}^{\mathrm{alg}}(\sA, \sJ)
&
&
HC_{n-1+2m}(\sA, \sJ)
\\
K_{n-1}^{\mathrm{alg}}(\sJ)
\ar[u]
\ar[r]
\ar[rr, "\ch_{n-1}^{(m)}"', bend right=20]
&
KH_{n-1}(\sJ)
\ar[r, "\ch_{n-1}^{(m)}"]
&
HC_{n-1+2m}(\sJ).
\ar[u]
\end{tikzcd}	
\end{equation}
If $\sA$ is a Banach algebra, one may define the \emph{topological cyclic homology} groups $HC^{\mathrm{top}}_n(\sA)$, $n \in \Z$ of $\sA$. 
These may be taken to be the homology groups of the cyclic complex $C^{\mathrm{top},\lambda}_\bullet(\sA)$, defined just as $C^\lambda_\bullet(\sA)$ but replacing the algebraic tensor product by the completed projective tensor product. 
The inclusion from the algebraic to the completed tensor product induces comparison maps $HC_n(\sA) \to HC_n^{\mathrm{top}}(\sA)$.
One also has a topological version of the Chern characters defined on topological $K$-theory, which has the property that the squares
\begin{equation*}
\begin{tikzcd}
K_n^{\mathrm{alg}}(\sA) \ar[r, "\ch_n^{(m)}"]
\ar[d]
&
HC_{n+2m}(\sA)
\ar[d]
\\
K_n^{\mathrm{top}}(\sA) \ar[r, "\ch_n^{(m)}"]
&
HC_{n+2m}^{\mathrm{top}}(\sA)
\end{tikzcd}
\end{equation*}
commute for any Banach algebra $\sA$ \cite[\S4.2]{KaroubiHomologieCyclique}.
It will be important that under the topological Chern character, Bott periodicity is mapped to the $S$-operator \eqref{Soperator} in the sense that we have the commutative diagram
\begin{equation}
\label{CompatibilityBottS}
\begin{tikzcd}[column sep=1.5cm]
K_n^{\mathrm{top}}(\sA) \ar[r, "\ch_n^{(m)}"]
\ar[d, "\beta"']
&
HC_{n+2m}^{\mathrm{top}}(\sA)
\\
K_{n+2}^{\mathrm{top}}(\sA) \ar[r, "\ch_{n+2}^{(m)}"]
&
HC_{n+2m+2}^{\mathrm{top}}(\sA)
\ar[u, "S/2\pi i"']
\end{tikzcd}
\end{equation}
see \cite[Corollaire 4.17]{KaroubiHomologieCyclique}.

\paragraph{Explicit formulas.}
We will use the following explicit cocycle representative for the Chern character in the degrees $n \in \{0, 1\}$, for unital $\sA$.
For any $n=2m$ even, then the Chern character $K_0^{\mathrm{alg}}(\sA) \to HC_{2m}(\sA)$ is represented by the explicit cocycle
\begin{equation}
\label{Ch0Formula}
  \ch_0^{(m)}(P) := (-1)^m \frac{(2m)!}{m!} \cdot  [\tr(\underbrace{P \otimes \cdots \otimes P}_{2m+1})].
\end{equation}
for idempotents $P = (P_a^b) \in M_k(\sA)$, where
\[
\tr(P \otimes \cdots \otimes P) = \sum_{a_0, \dots, a_n = 1}^k P_{a_0}^{a_1} \otimes P_{a_1}^{a_2} \otimes  \cdots \otimes P_{a_{n-1}}^{a_n} \otimes P_{a_n}^{a_0}.
\]
One may check that choosing a similar idempotent $P' = UPU^{-1}$ results in a cohomologous cocycle, hence $\ch_0^{(m)}$ is well-defined on $K_0^{\mathrm{alg}}(\sA)$.
%
%
The compatibility of these cocycles with the $S$-operators \eqref{CompatibilityS} may be explicitly checked on cocycle level using the description of the $S$-operator from \cite[Thm.~2.2.7]{Loday}.

For any $n = 2m+1$ odd, the Chern character $K_1^{\mathrm{alg}}(\sA) \to HC_n(\sA)$ of an invertible element $U \in \GL_k(\sA)^+$ has the explicit cocycle representation

\begin{equation}
\label{Ch1Formula}
\ch_1^{(m)}(U) = m! \cdot \big[\tr\big(\underbrace{(U - 1) \otimes (U^{-1} - 1) \otimes \cdots \otimes (U - 1) \otimes (U^{-1} - 1)}_{m+1 ~\text{pairs}}\big)\big].
\end{equation}
The compatibility with the $S$-operator may again be explicitly checked on cocycle level.
Also the commutativity of the diagrams \eqref{ChernVsBoundary} may be checked explicitly in the $n=1$ case using the explicit description of the boundary map as in \cite[\S4]{Milnor}.

\paragraph{Compatibility with the Mayer-Vietoris map.}

We will crucially use the following compatibility result between Mayer-Vietoris maps in $KH$ and in coarse homology.
\begin{lemma}
\label{LemmaChernMVboundary}
Let $\cX$, $\cY$ be big families in $M$.
Then for any $m, n \in \N_0$, $n \geq 1$, the following diagram commutes:
\[
\begin{tikzcd}
KH_{n-2m}(\sB(\cX \Cup \cY))
\ar[d, "\partial_{\mathrm{MV}}"']
\ar[r, "\ch^{(m)}_{n-2m}"] 
&[1cm]
HC_{n}(\sB(\cX \Cup \cY)) 
\ar[r, "\chi_*"]
&
HX_{n}(\cX \Cup \cY)
\ar[d, "\partial_{\mathrm{MV}}"]
\\
KH_{n-2m-1}(\sB(\cX \Cap \cY))
\ar[r, "\ch^{(m)}_{n-2m-1}"] 
&
HC_{n-1}(\sB(\cX \Cap \cY)) 
\ar[r, "\chi_*"]
&
HX_{n-1}(\cX \Cap \cY)
\end{tikzcd}
\]
\end{lemma}

\begin{proof}
Consider the diagram
\[
\begin{tikzcd}[column sep=2cm]
KH_{n-2m}(\sB(\cX \Cup \cY))
\ar[dddd, "\partial_{\mathrm{MV}}"', bend right=70]
\ar[d]
\ar[r, "\ch^{(m)}_{n-2m}"]
&
HC_n(\sB(\cX \Cup \cY))
\ar[d]
\\
\displaystyle KH_{n-2m}\left(\frac{\sB(\cX \Cup \cY)}{\sB(\cY)}\right)
\ar[r, "\ch^{(m)}_{n-2m}"]
&
\displaystyle HC_n\left(\frac{\sB(\cX \Cup \cY)}{\sB(\cY)}\right)
\\
\displaystyle KH_{n-2m}\left(\frac{\sB(\cX)}{\sB(\cX \Cap \cY)}\right)
\ar[u, "\cong"]
\ar[dd, "\partial"]
\ar[r, "\ch^{(m)}_{n-2m}"]
&
\displaystyle HC_n\left(\frac{\sB(\cX)}{\sB(\cX \Cap \cY)}\right)
\ar[u, "\cong"]
\ar[d, "\partial"]
\\
& HC_{n-1}\big(\sB(\cX), \sB(\cX \Cap \cY)\big)
\\
KH_{n-2m-1}(\sB(\cX \Cap \cY))
\ar[r, "\ch^{(m)}_{n-2m-1}"]
&
HC_{n-1}(\sB(\cX\Cap \cY)),
\ar[u]
\end{tikzcd}
\]
which commutes by the naturality of the Chern character and its compatibility \eqref{ChernVsBoundary} with boundary maps.
The left vertical composition is by definition the Mayer-Vietoris boundary.
We now paste this diagram to the following diagram,
\[
\begin{tikzcd}
HC_n(\sB(\cX \Cup \cY))
\ar[d]
\ar[r, "\chi_*"]
&
HX_n(\cX \Cup \cY)
\ar[d]
\ar[dddd, "\partial_{\mathrm{MV}}", bend left=70]
\\
\displaystyle HC_n\left(\frac{\sB(\cX \Cup \cY)}{\sB(\cY)}\right)
\ar[r, "\chi_*"]
&
HX_n(\cX \Cup \cY, \cY)
\\
\displaystyle HC_n\left(\frac{\sB(\cX)}{\sB(\cX \Cap \cY)}\right)
\ar[u, "\cong"]
\ar[d, "\partial"]
\ar[r, "\chi_*"]
&
HX_n(\cX, \cX \Cap \cY)
\ar[u, "\cong"]
\ar[dd, "\partial"]
\\
HC_{n-1}\big(\sB(\cX), \sB(\cX \Cap \cY)\big)
\ar[dr, "\chi_*", bend left=10]
\\
HC_{n-1}(\sB(\cX \Cap \cY)),
\ar[u]
\ar[r, "\chi_*"]
&
HX_{n-1}(\cX \Cap \cY),
\end{tikzcd}
\]
which commutes by Lemma~\ref{CharacterBoundary}.
The right vertical map is the Mayer-Vietoris boundary in coarse homology, see Remark~\ref{RemarkAlternativeMVboundary}.
This finishes the proof.
\end{proof}

\section{Proof of the quantization result}

\label{ProofQuantization}

In this section, we give a proof of the quantization result Thm.~\ref{MainQuantizationThm}.
We start by explaining how coarsely transverse collections of half-spaces and partitions give rise to coarse cohomology classes, to connect them to the Chern character and the coarse character map.
We then give a proof for the Kubo pairing via dimensional reduction and then reduce the Kitaev pairing to the Kubo case.

\subsection{Coarse cohomology classes from half-spaces}
\label{SectionHalfSpaceClasses}

Let $X_1, \dots, X_n$ be a coarsely transverse collection of half-spaces in a proper metric space $M$. 
For $k=0,\ldots,n-1$, we define the big families
\begin{equation}
\label{DefinitionXk}
\cX_k := \partial X_{k+1} \Cap \cdots \Cap \partial X_n,	
\end{equation}
so that $\cX_0 = \partial X_1 \Cap \cdots \Cap \partial X_n$   is the big family of bounded subsets (because the collection is coarsely transverse). 
By convention, we also set $\cX_n := \{M\}$.

\begin{lemma}\label{LemShortenedSwitchCochain}
For each $0 \leq k \leq n$, we have
\[
\mathbf{1} \wedge X_1 \wedge \cdots \wedge X_k \in CX^k_\alpha(\cX_k).
\]
\end{lemma}

\begin{proof}
We note that for each $i=1, \dots, k$, we may split up $\mathbf{1} = X_i + X_i^c$, hence by anti-symmetry, we have
\begin{equation}
\label{TrickyRewritten}
 \mathbf{1} \wedge X_1 \wedge \cdots \wedge X_k = X_i^c \wedge X_1 \wedge \cdots \wedge X_k.
\end{equation}
Let $R>0$ and take 
\[
x = (x_0, \dots, x_k) ~~\in~~ \supp(\mathbf{1} \wedge X_1 \wedge \cdots \wedge X_k) \cap \Delta_R \cap \left(\bigcap_{i={k+1}}^n (X_i)_R \cap (X_i^c)_R\right)^{k+1}.
\]
By \eqref{TrickyRewritten}, for each fixed $i=1, \dots, k$, there exists an index $0 \leq a \leq k$ such that $x_a \in X_i^c$ and another index $b$ such that $x_b \in X_i$.
But because $x \in \Delta_R$, we have $d(x_a, x_c) \leq 2R$ and $d(x_b, x_c) \leq 2R$ for all $0 \leq c\leq k$.
In total, we conclude that $x_c \in (X_i)_{2R} \cap (X_i^c)_{2R}$. 
We conclude that each coordinate $x_c$ is in fact contained in the set
\[
\left(\bigcap_{i=1}^k (X_i)_{2R} \cap (X_i^c)_{2R}\right) \cap \left(\bigcap_{i=k+1}^n (X_i)_{2R} \cap (X_i^c)_{2R}\right) ,
\]
which is bounded because $X_1, \dots, X_n$ are coarsely transverse.
\end{proof}

Observe that for each $1 \leq k \leq n$, we have a Mayer-Vietoris decomposition
\begin{align*}
\cX_{k} &= \big(\cX_{k} \Cap \{X_k\} \big) \Cup \big(\cX_{k} \Cap \{X_k^c\}  \big), \\
 \cX_{k-1} &= \big( \cX_{k} \Cap \{X_k\} \big) \Cap \big( \cX_{k} \Cap \{X_k^c\}  \big),
\end{align*}
so we obtain Mayer-Vietoris boundary maps
\[
\delta_{\mathrm{MV}} : HX^{\bullet-1}(\cX_{k-1}) \longrightarrow HX^{\bullet}(\cX_k), \rlap{\qquad $k=1, \dots, n$.}
\]
In particular, $HX^0(\cX_0) \cong \C$, the identification given by mapping $1 \in \C$ to the class of the constant function $\mathbf{1}$, so via an iterated application of the Mayer-Vietoris boundary map, we may iteratively define coarse cohomology classes $[\theta_k] \in HX^k(\cX_k)$, $k=0, \dots, n$,  by
\begin{equation}
\label{CompatibilityThetaKboundary}
[\theta_k] = \delta_{\mathrm{MV}}[\theta_{k-1}], \qquad \theta_0 = \mathbf{1}.
\end{equation}

\begin{lemma}
\label{LemmaFormulaThetak}
For each $0 \leq k \leq n$, the above cohomology class can be represented by
\begin{equation}
\label{ThetaCocycle}
\theta_k = (-1)^{\frac{k(k+1)}{2}}\frac{1}{k!} \cdot \mathbf{1} \wedge X_1 \wedge \cdots \wedge X_k.
\end{equation}
\end{lemma}

\begin{proof}
We prove the statement of the lemma by induction. 
For $k=0$, there is nothing to show.
Suppose now that we want to establish the result for $k \geq 1$, knowing the result for $k-1$.
Now by Remark~\ref{RemarkAntiSymmetricVersion}, we have $\delta_{\mathrm{MV}}[\theta_{k-1}] = [\delta \tilde{\theta}_{k-1}]$ with $\tilde{\theta}_{k-1}$ given by formula \eqref{AntisymmetricLocalization}. Explicitly, we have
\begin{align*}
\tilde{\theta}_{k-1} 
&= \frac{(-1)^{\frac{k(k-1)}{2}}}{k!} \cdot \left(\sum_{i=0}^{k-1} {\underbrace{\mathbf{1} \otimes \cdots \otimes \mathbf{1}}_i} \otimes X_k^c \otimes {\underbrace{\mathbf{1} \otimes \cdots \otimes \mathbf{1}}_{k-1-i}}\right)\cdot \mathbf{1} \wedge X_1 \wedge \cdots \wedge X_{k-1}
\\
&= \frac{(-1)^{\frac{k(k-1)}{2}}}{k!} \cdot \left(X_k^c \wedge X_1 \wedge \cdots \wedge X_{k-1} + \sum_{i=1}^{k-1} \mathbf{1} \wedge X_1 \wedge \cdots \wedge X_k^c X_i \wedge \cdots \wedge X_{k-1} \right).
\end{align*}
Then, using \eqref{SimpleASDifferential}, we get
\begin{align*}
[\theta_k]=\delta_{\mathrm{MV}}[\theta_{k-1}] 
= [\delta \tilde{\theta}_{k-1}]
&= \frac{(-1)^{\frac{k(k-1)}{2}}}{k!} \cdot \big[\mathbf{1} \wedge X_k^c \wedge X_1 \wedge \cdots \wedge X_{k-1}\big]
\\
&= \frac{(-1)^{\frac{k(k-1)}{2}}}{k!} (-1)^k\cdot \big[\mathbf{1} \wedge X_1 \wedge \cdots \wedge X_{k-1} \wedge X_k\big],
\end{align*}
which is the claimed result \eqref{ThetaCocycle}.
\end{proof}

We finish this section with a lemma for converting switch functions to half-spaces.

\begin{lemma}\label{LemSwitchToHalfSpace}
Let $\chi_1,\ldots,\chi_n$ be a coarsely transverse collection of switch functions on $M$. 
Then $\mathbf{1} \wedge \chi_1 \wedge \cdots \wedge \chi_n$ is a coarse cochain on $M$, and
\[
\big[\mathbf{1}\wedge\chi_1\wedge\dots\wedge \chi_n\big]=\big[\mathbf{1}\wedge X_1\wedge\cdots\wedge X_n\big] \in HX^n(M),
\]
where $X_i=\supp(\chi_i)$, $i=1,\ldots,n$.
\end{lemma}

\begin{proof}
Setting $W_i := \chi_i - X_i$ ($i=1, \dots, n$) and using \eqref{SimpleASDifferential}, one obtains the cochain identity
\begin{align*}
&\mathbf{1}\wedge \chi_1\wedge \dots\wedge\chi_{i}\wedge X_{i+1} \wedge\dots \wedge X_n - \mathbf{1}\wedge \chi_1\wedge \dots\wedge \chi_{i-1}\wedge X_i \wedge\dots \wedge X_n\\
& =\delta\big(
\chi_1\wedge \dots \wedge \chi_{i-1}\wedge W_i\wedge X_{i+1}\wedge \dots \wedge X_n
\big)
\\
& =\delta\Bigg(
\sum_{j=1}^n\big({\underbrace{\mathbf{1} \otimes \cdots \otimes \mathbf{1}}_{j-1}} \otimes W_i \otimes {\underbrace{\mathbf{1} \otimes \cdots \otimes \mathbf{1}}_{n-j}}\big)\cdot
\big(
\chi_1\wedge\dots\wedge \chi_{i-1}\wedge \mathbf{1}\wedge X_{i+1}\wedge \dots\wedge X_n
\big)
\Bigg),
\end{align*}
where in the second step, we use that all summands still containing a factor of $\mathbf{1}$ vanish after applying $\delta$, in light of \eqref{SimpleASDifferential}.

Now, the proof of Lemma~\ref{LemShortenedSwitchCochain} generalizes to show that $\mathbf{1}\wedge \chi_1\wedge \dots\wedge\chi_{i-1}\wedge X_{i+1} \wedge\dots \wedge X_n$ is contained in $CX^{n-1}_\alpha(\cY_i)$ with
\[
\cY_i = \{\supp(\chi_i)\} \Cap \{\supp(\mathbf{1} - \chi_i)\}.
\]
Since $\mathrm{supp}(W_i) \in \cY_i$, Lemma~\ref{LemmaLocalization} (and Remark~\ref{RemarkAntiSymmetricVersion}) yields that the term in the parentheses above is contained in $CX^{n-1}_\alpha(M)$.
Therefore, one obtains that if $\mathbf{1}\wedge \chi_1\wedge \dots\wedge\chi_{i}\wedge X_{i+1} \wedge\dots \wedge X_n$ is a coarse cochain on $M$, then so is $\mathbf{1}\wedge \chi_1\wedge \dots\wedge\chi_{i-1}\wedge X_{i} \wedge\dots \wedge X_n$, and both define the same cohomology class.
This implies the result by induction on $i$.
\end{proof}

\subsection{Coarse cohomology classes from partitions}
\label{SectionPartitionClasses}

Let $M$ be a proper metric space.

\begin{lemma}
\label{LemmaPartitionCochain}
Let $A_0, \dots, A_n \subseteq M$ be a collection of Borel subsets.
\begin{enumerate}[{\normalfont(1)}]
\item	
If $\cY$ is a big family such that 
\begin{equation}
\label{CoarseTransversalityOfAi}
\text{all members of} \qquad \{A_0\} \Cap \cdots \Cap\{A_n\} \Cap \cY\qquad \text{are bounded},
\end{equation}
then the cochain $A_0 \wedge\cdots \wedge A_n$ is contained in $CX^n_\alpha(\cY)$.
\item
If the subsets form a partition, then $A_0 \wedge\cdots \wedge A_n$ is Alexander-Spanier closed.
\end{enumerate}
\end{lemma}

In particular, it follows from this lemma that if $A_0, \dots, A_n$ is a coarsely transverse partition, so \eqref{CoarseTransversalityOfAi} holds for $\cY=\{M\}$, then $A_0 \wedge \cdots \wedge A_n$ is a coarse cocycle on all of $M$, hence defines a coarse cohomology class
\[
[A_0 \wedge \cdots \wedge A_n] \in HX^n(M).
\] 

\begin{proof}
(1) Suppose that \eqref{CoarseTransversalityOfAi} holds.
 For some given $R>0$ and a member $Y$ of $\cY$, let 
\[x = (x_0, \dots, x_n) \in \supp(A_0 \wedge \cdots \wedge A_n) \cap\Delta_R\cap Y^{n+1}.
\]
For each $i =0, \dots, n$, there exists at least one $x_a$ with $x_a \in A_i$.
Since $x \in \Delta_R$, we have $d(x_b, x_a) \leq 2R$ for any $b$, so we have $x_b \in (A_i)_{2R}$. 
Because $i$ was arbitrary, we obtain that $x_b \in (A_0)_{2R} \cap \cdots \cap (A_n)_{2R}\cap Y$ for each $b$.
This shows that $\supp(A_0 \wedge \cdots \wedge A_n) \cap \Delta_R \cap Y^{n+1}$ is bounded.

(2) If the sets $A_i$ form a partition, then using the formula  \eqref{SimpleASDifferential} for $\delta$ and anti-symmetry, we get
\[
\delta(A_0 \wedge \cdots \wedge A_n) 
= \mathbf{1} \wedge A_0 \wedge \cdots \wedge A_n
= \sum_{i=0}^n A_i \wedge A_0 \wedge \cdots \wedge A_n = 0. \qedhere
\]
\end{proof}

We are interested in the relation between cohomology classes corresponding to coarsely transverse partitions and those corresponding to collections of coarsely transverse half-spaces.
We will show below that each half-space class is actually an integer multiple of a partition class.

Generalizing Example~\ref{ExampleEuclideanHalfSpaces}, the following prescription produces a coarsely transverse partition from any given  coarsely transverse collection of half-spaces.

\begin{lemma}\label{LemPartitionFromHalfSpaces}
If $X_1, \ldots, X_n \subseteq M$ are arbitrary Borel subsets of $M$, then the sets
\begin{align*}
A_0 & = X_1\cdots X_n,\\
A_i & = X_i^c X_{i+1}\cdots X_n,\qquad i=1,\ldots,n,
\end{align*}
form a partition of $M$.
Moreover, if the $X_i$ form a  coarsely transverse collection of half-spaces, then this partition is coarsely transverse.
\end{lemma}

\begin{proof}
The $A_i$ defined above indeed form a partition due to the telescoping sum
\begin{align*}
\sum_{i=0}^n A_i &=X_1 X_2 X_3\cdots X_n + X_1^c X_2 X_3\cdots X_n + X_2^c X_3\cdots X_n\dots + X_{n-1}^c X_n+X_n^c\\
&=X_2 X_3\cdots X_n+X_2^c X_3\cdots X_n+\dots + X_{n-1}^c X_n+X_n^c\\
& ~~\vdots
\\
&= X_n+X_n^c=\mathbf{1}.
\end{align*}
For any $R>0$, we have
\begin{align*}
\bigcap_{i=0}^n(A_i)_R
&=\big(X_1 X_2\cdots X_n\big)_R\cap \big(X_1^c X_2\cdots X_n\big)_R\cap\dots \cap \big(X_n^c\big)_R\\
&\subseteq \big(X_1\big)_R\cap \dots \cap \big(X_n\big)_R\cap \big(X_1^c\big)_R\cap \dots \cap \big(X_n^c\big)_R.
\end{align*}
It follows that if $X_1, \ldots, X_n$ are coarsely transverse half-spaces, then $A_0,\ldots, A_n$ are coarsely transverse subsets.
\end{proof}

Similarly, for each permutation $\sigma\in S_n$, we have a partition 
\begin{equation}
\label{eqn:PermutedPartitions}
\begin{aligned}
A^\sigma_0 & = X_{\sigma_1}X_{\sigma_2}\cdots X_{\sigma_n}, \\
A^\sigma_i & = X_{\sigma_i}^cX_{\sigma_{i+1}}\cdots X_{\sigma_n}, 
\rlap{\qquad $(i=1,\ldots,n)$,}
\end{aligned}
\end{equation}
of $M$, which is coarsely transverse if $X_1,\ldots,X_n$ are coarsely transverse half-spaces. 
In that case, $A^\sigma_0\wedge \dots \wedge A^\sigma_n$ is a coarse $n$-cocycle on $M$.

\begin{proposition}[Sum rule]\label{PropSumOfPartitions}
For coarsely transverse half-spaces $X_1,\ldots, X_n$ in $M$, the formula
\[
[\mathbf{1}\wedge X_1\wedge \dots \wedge X_n]=(-1)^n\cdot \sum_{\sigma\in S_n}\sgn(\sigma)\cdot[A^\sigma_0\wedge \dots \wedge A^\sigma_n] \]
holds in $HX^n(M)$.
\end{proposition}

%
%
%

\begin{proof}
We will prove this proposition by induction.
To this end, we more generally define for each permutation $\sigma \in S_k$ with $1 \leq k \leq n$,
\begin{equation}
\label{eqn:PartitionSubpermutation}
\begin{aligned}
A^\sigma_0 & = X_{\sigma_1}X_{\sigma_2}\cdots X_{\sigma_k}, \\
A^\sigma_i & = X_{\sigma_i}^cX_{\sigma_{i+1}}\cdots X_{\sigma_k}, 
\rlap{\qquad $(i=1,\ldots, k)$.}
\end{aligned}
\end{equation}
Note that although $A^\sigma_0,\ldots, A^\sigma_k$ is a partition of $M$, it may not be coarsely transverse unless $k=n$. Rather, we have
\[
\{A^\sigma_0\} \Cap \cdots \Cap \{A^\sigma_k\}
=
 \partial X_1\Cap\dots \Cap \partial X_k.
\]
By Lemma~\ref{LemmaPartitionCochain}, this implies that $A^\sigma_0\wedge \dots \wedge A^\sigma_k$ is contained in $CX^k_\alpha(\cX_k)$, 
with $\cX_k$ as in \eqref{DefinitionXk}.
Moreover, this element is closed, hence defines a class in $HX^k(\cX_k)$. 
We will now prove the more general statement that for each $k=1, \dots, n$, we have the formula
\begin{equation}\label{eqn:SumTheoremFormula}
[\mathbf{1}\wedge X_1\wedge \dots \wedge X_k]=(-1)^k\cdot \sum_{\sigma\in S_k}\sgn(\sigma)\cdot[A^\sigma_0\wedge \dots \wedge A^\sigma_k] \in HX^k(\cX_k).
\end{equation}
Because $\cX_n=\{M\}$, the $k=n$ case is precisely the statement of the Proposition.

For $k=1$, \eqref{eqn:SumTheoremFormula} is automatic, since $A^{(1)}_0=X_1$ and $A^{(1)}_1=X_1^c$, so
\[
\mathbf{1}\wedge X_1=(X_1+X_1^c)\wedge X_1=X_1^c\wedge X_1=-A^{(1)}_0\wedge A^{(1)}_1 \in CX_\alpha^1(\cX_1),
\]
already at the level of cochains. 
%

For general $k\geq 2$, direct comparison of the cochains in \eqref{eqn:SumTheoremFormula} quickly becomes much more difficult, and we shall proceed by induction. 
We appeal to Lemma~\ref{LemmaFormulaThetak}, which implies that
\begin{equation}
\begin{aligned}
\label{eqn:SumTheoremMV}
\frac{(-1)^{k}}{k} \cdot [\mathbf{1}\wedge X_1\wedge\dots\wedge X_k]
&=\delta_{\mathrm{MV}}\big([\mathbf{1}\wedge X_1 \wedge \dots \wedge X_{k-1}]\big)
\\
&= (-1)^{k-1} \cdot \sum_{\sigma \in S_{k-1}} \sgn(\sigma) \cdot \delta_{\mathrm{MV}}\big([A_0^\sigma \wedge \cdots \wedge A_{k-1}^\sigma]\big)
\end{aligned}
\end{equation}
in $HX^k(\cX_k)$, where in the second step, we substituted the induction hypothesis.
By Remark~\ref{RemarkAntiSymmetricVersion}, the class $\delta_{\mathrm{MV}}([A_0^\sigma \wedge \cdots \wedge A_{k-1}^\sigma])$ may be represented by the cocycle
\begin{align*}
&-\delta\left( \left(\frac{1}{k} \cdot\sum_{i=0}^{k-1} {\underbrace{\mathbf{1} \otimes \cdots \otimes \mathbf{1}}_i} \otimes X_k \otimes {\underbrace{\mathbf{1} \otimes \cdots \otimes \mathbf{1}}_{k-1-i}}\right)\cdot A_0^\sigma \wedge \cdots \wedge A_{k-1}^\sigma\right)
\\
&\hspace{4cm}= -\frac{1}{k} \sum_{i=0}^{k-1}\mathbf{1}\wedge A^\sigma_0\wedge\dots\wedge A^\sigma_{i-1}\wedge A^{\sigma}_i X_k\wedge A^\sigma_{i+1}\wedge\dots A^\sigma_{k-1}.
\end{align*}
For each $j\neq i$, we may replace each occurrence of $A_j^\sigma$ in this formula by $A^\sigma_j X_k + A^\sigma_j X_k^c$. 
However, it follows from Lemma~\ref{LemmaPartitionCochain}  that for each $j \neq i$, the cochain
\[
A^\sigma_0\wedge \dots \wedge A_j^\sigma X_k^c \wedge\dots\wedge A^\sigma_{i-1}\wedge A^{\sigma}_iX_k\wedge A^\sigma_{i+1}\wedge\dots A^\sigma_{k-1}
\]
is in fact contained in $CX^k_\alpha(\cX_k)$, hence the corresponding term is exact.
%
%
Therefore, at the level of cohomology, we may simply replace all the $A^\sigma_j$ with $A^\sigma_j X_k$, which yields
\begin{align*}
\delta_{\mathrm{MV}}\big([A_0^\sigma \wedge \cdots \wedge A_{k-1}^\sigma]\big)
= -\left[\mathbf{1}\wedge A^\sigma_0 X_k\wedge\dots\wedge A^\sigma_{k-1} X_k\right].
\end{align*}

In view of formula \eqref{eqn:PartitionSubpermutation}, we have $A_j^\sigma X_k = A_j^{\sigma^+}$, where $\sigma^+\in S_k$ is the extended permutation
\[
\sigma^+ = \begin{pmatrix}
1 & 2 & 3 & \cdots & k-1 & k
\\
\sigma_1 & \sigma_2 & \sigma_3 & \cdots & \sigma_{k-1} & k
\end{pmatrix}.
\]
Notice that $\sgn(\sigma^+) = \sgn(\sigma)$. 
Hence, with this notation, \eqref{eqn:SumTheoremMV} becomes
\begin{align}
\frac{(-1)^k}{k}\cdot\big[\mathbf{1}\wedge X_1\wedge \dots \wedge X_k\big] &=(-1)^{k-1}\cdot\!\!\!\sum_{\sigma\in S_{k-1}}\sgn(\sigma^+)\cdot\Big(-\left[\mathbf{1}\wedge A^{\sigma^+}_0\wedge\dots\wedge A^{\sigma^+}_{k-1}\right]\Big) \nonumber\\
&=(-1)^k\cdot\!\!\!\sum_{\sigma\in S_{k-1}}\sgn(\sigma^+)\cdot\left[\Big(\sum_{i=0}^k A^{\sigma^+}_i\Big)\wedge A^{\sigma^+}_0\wedge\dots\wedge A^{\sigma^+}_{k-1}\right]
\nonumber\\
&=\sum_{\sigma\in S_{k-1}}\sgn(\sigma^+)\cdot\left[A^{\sigma^+}_0\wedge\dots\wedge A^{\sigma^+}_{k-1}\wedge A^{\sigma^+}_k\right]
\nonumber\\
&=\sum_{\substack{\sigma\in S_k \\\sigma_k=k}}\sgn(\sigma)\cdot\big[A^\sigma_0\wedge\dots\wedge A^\sigma_k\big].\label{eqn:SumTheoremSubpermutation}
\end{align}

Now, for each $\ell=1,\ldots,k$, let $\tau^{(\ell)}\in S_k$ be the cyclic permutation 
\[
\tau^{(\ell)}
= \begin{pmatrix}
1 & \dots & k-\ell & k-\ell+1 & \dots & k
\\
	\ell+1 & \ldots & k &1 &\ldots &\ell
\end{pmatrix}.
\]
From \eqref{eqn:SumTheoremSubpermutation}, we further conclude that
\begin{align*}
\frac{(-1)^k}{k}\cdot[\mathbf{1}\wedge X_1\wedge \dots \wedge X_k]
&=\frac{(-1)^k}{k}\cdot \sgn(\tau^{(\ell)})\cdot\big[\mathbf{1}\wedge X_{\tau^{(\ell)}_1}\wedge\dots \wedge X_{\tau^{(\ell)}_k}\big]\\
&=\sum_{\substack{\sigma\in S_k \\ \sigma_k = k}}\sgn(\tau^{(\ell)}\circ\sigma )\cdot\left[A^{\tau^{(\ell)}\circ \sigma}_0\wedge\dots\wedge A^{\tau^{(\ell)}\circ \sigma}_k\right]\\
&=\sum_{\substack{\sigma\in S_k \\ \sigma_k=\ell}}\sgn(\sigma)\cdot\left[A^\sigma_0\wedge\dots\wedge A^\sigma_k\right]
\end{align*}
holds for each $\ell=1,\ldots,k$ as well.
 Summing over $\ell=1,\ldots,k$ gives the claimed formula \eqref{eqn:SumTheoremFormula}.
\end{proof}

\begin{proposition}[Equipartition]\label{PropEquipartition}
Let $X_1,\ldots, X_n$ be coarsely transverse half-spaces in $M$. 
For any $\sigma\in S_{n}$, we have
\[
[A^\sigma_0\wedge \dots \wedge A^\sigma_n]=\sgn(\sigma)\cdot[A_0\wedge \dots \wedge A_n]\in HX^n(M).
\]
\end{proposition}

\begin{proof}
Fix any $2 \leq i \leq n$, and consider the adjacent transposition $(i-1,i)$. 
First, note that $A^{(i-1,i)}_0 \wedge \dots \wedge A^{(i-1,i)}_n$ differs from $A_0 \wedge \dots \wedge A_n$ only in the $(i-1)$-th and $i$-th wedge factors. 
Due to anti-symmetry and $\sum_{j=0}^n A_j=\mathbf{1}$, we can replace their $i$-th wedge factors by $\mathbf{1}$, so that
\begin{align*}
& A_0 \wedge \dots \wedge A_n +  A^{(i-1,i)}_0 \wedge \dots \wedge A^{(i-1,i)}_n \\
&= A_0 \wedge \cdots \wedge A_{i-1} \wedge \mathbf{1} \wedge A_{i+1} \wedge \cdots \wedge A_n + A_0 \wedge \cdots \wedge A_{i-1}^{(i-1, i)} \wedge \mathbf{1} \wedge A_{i+1} \wedge \cdots \wedge A_n
\\
&= A_0 \wedge \dots \wedge A_{i-2} \wedge  \left( ({X_{i-1}^cX_{i}} +X_{i}^cX_{i-1})X_{i+1}\cdots X_n\right)\wedge \mathbf{1} \wedge A_{i+1} \wedge\dots \wedge A_n.
\end{align*}
Next, we add the term
\begin{align*}
&A_0 \wedge \dots \wedge A_{i-2} \wedge   (X_{i-1}^c X_i^c X_{i+1}\cdots X_n)\wedge \mathbf{1} \wedge A_{i+1} \wedge\dots \wedge A_n,
\end{align*}
which is exact because the sets
\begin{align*}
& X_1X_2\cdots X_n, \qquad X_1^cX_2\cdots X_n, \qquad\dots\qquad, X_{i-2}^cX_{i-1}\cdots X_n, \\
&~~~~X_{i-1}^c X_i^c X_{i+1}\dots X_n, \qquad X_{i+1}^cX_{i+2}\cdots X_n, \qquad \dots \qquad X_n^c,
\end{align*}
are coarsely transverse.
So in cohomology, 
\begin{align*}
& \big[A_0 \wedge \dots \wedge A_n\big] +  \big[A^{(i-1,i)}_0 \wedge \dots \wedge A^{(i-1,i)}_n\big] \\
&= \big[A_0 \wedge \dots  \wedge A_{i-2} \wedge\left(( {X_{i-1}^cX_{i}} +X_{i}^cX_{i-1} +X_{i-1}^cX_{i}^c) X_{i+1}\cdots X_n\right)\wedge \mathbf{1} \wedge A_{i+1}\wedge\dots \wedge A_n\big]\\
&=\big[(X_1X_2\cdots X_n) \wedge (X_2\cdots X_n) \wedge \dots \wedge (X_{i-1}\cdots X_n)\wedge\\
& \qquad\wedge \left( ({X_{i-1}^cX_{i}} +X_{i}^cX_{i-1}  +X_{i-1}^cX_{i}^c) X_{i+1}\cdots X_n\right)\wedge \mathbf{1} \wedge (X_{i+1}^cX_{i+2}\cdots X_n)\wedge\dots \wedge X_n^c\big],
\end{align*}
where in the second equality, we used anti-symmetry to rewrite
\begin{align*}
A_0 \wedge A_1 \wedge \dots \wedge A_{i-2}
&= A_0 \wedge (A_0 + A_1) \wedge (A_0 + A_1 + A_2) \wedge \dots \wedge (A_0 + \dots + A_{i-2})	
\end{align*}
and invoked the identity $A_0 + \dots + A_j = X_{j+1} \cdots X_n$ 
(which holds for any $j$).
To proceed, we observe that
\begin{align*}
&(X_{i-1}\cdots X_n)\wedge \big(({X_{i-1}^cX_{i}} +X_{i}^cX_{i-1}  +X_{i-1}^cX_{i}^c) X_{i+1}\cdots X_n\big)\\
&=(X_{i-1}\cdots X_n)\wedge \big((\underbrace{{X_{i-1}^cX_{i}} +X_{i}^cX_{i-1}  +X_{i-1}^cX_{i}^c + X_{i-1}X_i}_{=\mathbf{1}}) X_{i+1}\cdots X_n\big) \\
&=(X_{i-1}\cdots X_n) \wedge (X_{i+1}\cdots X_n),
\end{align*}
so we have the simplification
\begin{align*}
& \big[A_0 \wedge \dots \wedge A_n\big] +  \big[A^{(i-1,i)}_0 \wedge \dots \wedge A^{(i-1,i)}_n\big] \\
&=\big[(X_1X_2\cdots X_n) \wedge (X_2\cdots X_n) \wedge \dots \wedge (X_{i-1}\cdots X_n)\wedge\\
& \qquad\qquad \wedge\underbracket{(
 X_{i+1}\cdots X_n )\wedge \mathbf{1} \wedge (X_{i+1}^cX_{i+2}\cdots X_n)\wedge\dots \wedge X_n^c}\big].
\end{align*}
The bracketed factor now vanishes by another telescoping argument,
\begin{align*}
&(X_{i+1}\cdots X_n)\wedge \mathbf{1}\wedge (X_{i+1}^cX_{i+2}\cdots X_n)\wedge \dots \wedge (X_{n-1}^cX_n) \wedge X_n^c\\
&\quad =(X_{i+1}\cdots X_n)\wedge\mathbf{1}\wedge (X_{i+2}\cdots X_n)\wedge (X_{i+2}^c X_{i+3} \cdots X_n)\wedge \dots \wedge X_{n-1}^cX_n\wedge X_n^c\\
&\quad =(X_{i+1}\cdots X_n)\wedge\mathbf{1}\wedge (X_{i+2}\cdots X_n)\wedge \cdots \wedge X_n \wedge \mathbf{1} =0.
\end{align*}
We therefore obtain that
\[
[A_0 \wedge \dots \wedge A_n]+[A^{(i-1,i)}_0 \wedge \dots \wedge A^{(i-1,i)}_n] = 0.
\]
The same argument shows that starting from any permutation $\sigma\in S_n$, and applying any adjacent transposition $(i-1,i)$,
\[
[A^{\sigma\circ (i-1,i)}_0 \wedge \dots \wedge A^{\sigma\circ (i-1,i)}_n] = -[A^{\sigma}_0 \wedge \dots \wedge A^{\sigma}_n],\qquad i=2,\ldots,n.
\]
Since any permutation can be obtained as a sequence of adjacent transpositions, the Proposition follows.
\end{proof}

\begin{corollary}\label{CorEquipartition}
Let $X_1,\ldots,X_n$ be coarsely transverse half-spaces in $M$, and $A_0,\ldots,A_n$ be the associated coarsely transverse partition as given in Lemma~\ref{LemPartitionFromHalfSpaces}. Then
\[
[\mathbf{1}\wedge X_1\wedge\dots\wedge X_n]=(-1)^n n!\cdot[A_0\wedge \dots\wedge A_n]\in HX^n(M).
\]
\end{corollary}

\begin{proof}
This follows immediately from the sum rule and equipartition result, Propositions \ref{PropSumOfPartitions} and \ref{PropEquipartition} respectively.
\end{proof}

\subsection{Integrality result for Kubo pairings}
\label{SectionIntegralityKubo}

Let $M$ be a proper metric space and choose an $M$-module so that the algebra $\sB(M)$ of locally trace-class, finite propagation operators is defined.
We then have the following observation.

\vspace{-0.1cm}

\begin{lemma}
\label{LemmaPartitionInTermsOfChernCharacters}
Let $A_0, \dots, A_n$ be a coarsely transverse partition of $M$ and let $X_1, \dots, X_n$ be a coarsely transverse collection of half-spaces. 
Let $P=P^2\in M_k(\sB(M)^+)$ and $U\in \GL_k(\sB(M))^+\subseteq \GL_k(\sB(M)^+)$.
Then we have
\begin{align*}
[P; A_0, \dots, A_n] &= (-1)^m\frac{m!}{(2m)!}\cdot \langle \widetilde{\chi}_*\ch(P), A_0 \wedge \cdots \wedge A_n\rangle
& & (n=2m~\text{even})
\\
[U; A_0, \dots, A_n] 
&= \frac{1}{m!}\cdot \langle \chi_*\ch(U), A_0 \wedge \cdots \wedge A_n\rangle
& & (n=2m+1~\text{odd})
\\
[P; X_1, \dots, X_n] &= (-1)^m\frac{m!}{(2m)!}\cdot\langle \widetilde{\chi}_*\ch(P), \mathbf{1} \wedge X_1 \wedge \cdots \wedge X_n\rangle
& & (n=2m~\text{even})
\\
[U; X_1, \dots, X_n] &= \frac{1}{m!}\cdot\langle \chi_*\ch(U), \mathbf{1} \wedge X_1 \wedge \cdots \wedge X_n\rangle
& & (n= 2m+1~\text{odd})
\end{align*}
Here 
 $\chi_*$ is the character map \eqref{CharacterMap}, and the right hand side is the pairing \eqref{CoarsePairing} with the coarse cohomology class determined by the partition, respectively collection of half-spaces (see Section~\ref{SectionHalfSpaceClasses} \& \ref{SectionPartitionClasses}).
\end{lemma}

\begin{proof}
In the case of an idempotent $P \in M_k(\sB(M)^+)$, compare the formula \eqref{Ch0Formula} for the Chern character with the alternative formulas \eqref{AlternativeFormulaKitaevPairing} respectively \eqref{KuboPairingCommutatorProduct2}.
In the case of an invertible element $U \in \GL_k(\sB(M))^+$, compare \eqref{Ch1Formula} with the alternative formula \eqref{KitaevPairingAlt}, respectively \eqref{AlternativeFormulaKuboOdd2}.
\end{proof}

We now prove the quantization result, Thm.~\ref{MainQuantizationThm}, for the Kubo pairings.
To this end, let $X_1, \dots, X_n \subseteq M$ be a coarsely transverse collection of half-spaces.
Let $\theta_n$ be the corresponding coarse cocycle given by \eqref{ThetaCocycle}, which agrees with $\mathbf{1} \wedge X_1 \wedge \cdots \wedge X_n$ up to a sign.
Consider the composition
\begin{equation}
\label{ImageNeeded}
	K^{\mathrm{alg}}_{n-2m}(\sB(M)^+)
	\xrightarrow{~~\ch_{n-2m}^{(m)}~}
	HC_n(\sB(M)^+)
	\xrightarrow{~~\tilde{\chi}_*~~}
	HX_n(M)
\xrightarrow{~~\substack{\text{pairing} \\~\text{with}~\theta_n}~~}
	\C,
\end{equation}
where $m$ is chosen such that $n - 2m \in \{0, 1\}$.
In light of Lemma~\ref{LemmaPartitionInTermsOfChernCharacters}, this composition agrees with the Kubo pairing up to a dimensional factor.
Thus, it remains to  show that the image of the composition
is the subgroup 
\begin{equation}
\label{TargetSubgroup}
\frac{1}{(2 \pi i)^m} \cdot \Z \subset \C, 
\qquad
m=
\begin{cases}
\frac{n}{2} & n~\text{even},	
\\
\frac{n-1}{2} & n~\text{odd}.
\end{cases}
\end{equation}
The determination of the groups $K^{\mathrm{alg}}_\bullet(\sB(M)^+)$ or the coarse (co)homology groups of $M$ would be a formidable task. Fortunately, for the purposes of showing that the composition \eqref{ImageNeeded} takes values in \eqref{TargetSubgroup}, a dimensional reduction argument to a manageable case suffices, as detailed in the proof below.

\begin{proof}[Proof of \eqref{TargetSubgroup}]
We drop the indices on the relevant Chern characters for better readability.
For every $m, n \in \N_0$, we may consider the following diagram:
\[
	\begin{tikzcd}[column sep=1cm]
	K^{\mathrm{alg}}_{n-2m}(\sB(M)^+)
	\ar[d]
	\ar[dr, "\ch", bend left=15]
	& & &[0.8cm]
	\\
	KH_{n-2m}(\sB(M)^+)
	\ar[d, "\partial_{\mathrm{MV}}"]
	\ar[r, "\ch"]
	&
	HC_n(\sB(M)^+)
	\ar[r, "\widetilde{\chi}_*"]
	&
	HX_n(M)
	\ar[r, "\substack{\text{pairing} \\~\text{with}~\theta_n}"]
	\ar[d, "\partial_{\mathrm{MV}}"]
	&
	\C
	\ar[d, equal]
	\\
	KH_{n-2m-1}(\sB(\cX_{n-1}))
	\ar[r, "\ch"]
	\ar[d, "\partial_{\mathrm{MV}}"]
	&
	HC_{n-1}(\sB(\cX_{n-1}))
	\ar[r, "\chi_*"]
	&
	HX_{n-1}(\cX_{n-1})
	\ar[r, "\substack{\text{pairing} \\~\text{with}~\theta_{n-1}}"]
	\ar[d, "\partial_{\mathrm{MV}}"]
	&
	\C	
	\ar[d, equal]
	\\
	KH_{n-2m-2}(\sB(\cX_{n-2}))
	\ar[r, "\ch"]
	\ar[d, "\partial_{\mathrm{MV}}"]
	&
	HC_{n-2}(\sB(\cX_{n-2}))
	\ar[r, "\chi_*"]
	&
	HX_{n-2}(\cX_{n-2})
	\ar[r, "\substack{\text{pairing} \\~\text{with}~\theta_{n-2}}"]
	\ar[d, "\partial_{\mathrm{MV}}"]
	&
	\C	
	\ar[d, equal]
	\\
	\vdots
	\ar[d, "\partial_{\mathrm{MV}}"]
	& &
	\vdots 
	\ar[d, "\partial_{\mathrm{MV}}"]
	&
	\vdots
	\ar[d, equal]
	\\
	KH_{-2m}(\sB(\cX_0))
	\ar[r, "\ch"]
	&
	HC_0(\sB(\cX_0))
	\ar[r, "\chi_*"]
	&
	HX_0(\cX_0)
	\ar[r, "\substack{\text{pairing} \\~\text{with}~\theta_{0}}"]
	&
	\C
\end{tikzcd}
\]
Here the rectangles on the left hand side commute by Lemma~\ref{LemmaChernMVboundary}, while the right squares commute by \eqref{CompatibilityThetaKboundary} and \eqref{DualityMVdifferentials}. 
We mention that $\partial_{\mathrm{MV}}$ admits an obvious modification which starts at the unitized algebra $\sB(M)^+$ instead of $\sB(M)$ --- for example, the Mayer-Vietoris short exact sequence \eqref{eqn:KHSES} can have $\sB(\cX)^+$ in place of $\sB(\cX)$.
In the odd case, we don't need the unitization.

We aim to understand the image of the bottom horizontal composition in the above diagram.
To begin with, we note that  $\cX_0 = \cB$, the big family of bounded sets, so by Lemma~\ref{LemmaTraceClass}, we have $\sB(\cX_0)\subseteq \sL_1(\cH)$, the algebra of trace-class operators on the $M$-module $\cH$.
Moreover, since $\theta_0 = \mathbf{1}$ is constant on $M$, pairing with it sends a measure to its integral and the precomposition with the character map takes an element in $HC_0(\sB(\cX_0))$ to its trace. 
In other words, we have a commutative diagram
\[
\begin{tikzcd}[column sep=1.5cm]
	KH_{-2m}(\sB(\cX_0))
	\ar[r, "\ch"]
	\ar[d]
	&
	HC_0(\sB(\cX_0))
	\ar[r, "\chi_*"]
	\ar[d]
	&
	HX_0(\cX_0)
	\ar[r, "\substack{\text{pairing} \\~\text{with}~\theta_{0}}"]
	&
	\C
	\ar[d, equal]
	\\
	KH_{-2m}(\sL_1)
	\ar[r, "\ch"]
	\ar[d]
	&
	HC_0(\sL_1)
	\ar[rr, "\Tr"]
	&
	&
	\C
	\ar[d, equal]
	\\
	K^{\mathrm{top}}_{-2m}(\sL_1) 
	\ar[r, "\ch"]
	&
	HC_0^{\mathrm{top}}(\sL_1)
	\ar[rr, "\Tr"]
	\ar[u, equal]
	&
	&
	\C,
\end{tikzcd}
\]
where to get to the last line, we composed with the comparison map \eqref{FactorizationComparisonMap}. Here, we note that by \eqref{ZerothCyclicHomology}, we have $HC_0(\sL_1)=\sL_1/[\sL_1,\sL_1]$, so the trace map is well-defined on $HC_0(\sL_1)$ by cyclicity of trace.

%
To determine the image of the bottom map in the previous diagram, we use the Bott periodicity isomorphism $\beta$, which, under the Chern character, corresponds to the $S$-operator up to a factor of $2 \pi i$, see \eqref{CompatibilityBottS}.
This yields the commutative diagram
\[
\begin{tikzcd}
	K^{\mathrm{top}}_{-2m}(\sL_1) 
	\ar[r, "\ch"]
	\ar[d, "\beta^m"']
	&
	HC_0^{\mathrm{top}}(\sL_1)
	\\
	K^{\mathrm{top}}_{0}(\sL_1)
	\ar[r, "\ch"]
	&
	HC^{\mathrm{top}}_{2m}(\sL_1).
	\ar[u, "(S/2\pi i)^m"']
\end{tikzcd}
\]
Pasting all the previous diagrams, we see that the image of the map \eqref{ImageNeeded} is contained in the image of the composition
\begin{equation*}
\begin{tikzcd}[column sep=1.5cm]
K^{\mathrm{top}}_0(\sL_1) 
\ar[r, "\ch"]
\ar[rr, dashed, bend right=18, "(2 \pi i)^{-m} \cdot\, \ch"', near end]
&
HC_{2m}^{\mathrm{top}}(\sL_1) 
\ar[r, "(S/2 \pi i)^m"]
&
HC_0^{\mathrm{top}}(\sL_1) = \sL_1 /[\sL_1, \sL_1]
\ar[r, "\Tr"]
&
\C.
\end{tikzcd}
\end{equation*}
Here the composition of the first two arrows equals the dashed arrow by compatibility of $S$ with $\ch$, see \eqref{CompatibilityS}.
The entire map is just given by sending 
\[
K_0(\sL_1) \ni x = [P] - [Q] \longmapsto (2 \pi i)^{-m} \cdot \big(\Tr(P - Q)\big) \in \C.
\]
Since $\sL_1$ is holomorphically closed in the algebra $\sK$ of compact operators, the inclusion $\sL_1 \hookrightarrow \sK$ induces an isomorphism
\begin{equation}
\label{KtheoryCompactOperators}
K^{\mathrm{top}}_n(\sL_1)\cong K^{\mathrm{top}}_n(\sK) \cong \begin{cases}
 \Z	& \text{if}~n~\text{is even},
 \\
 0	& \text{if}~n~\text{is odd},
 \end{cases}
\end{equation}
so any class $x \in K_0^{\mathrm{top}}(\sL_1)$ may be represented as $x = [P] - [Q]$ with $P, Q$ finite-rank projections in $\sL_1$, by the well-known description of $K_0^{\mathrm{top}}(\sK)$.
Therefore $\Tr(P - Q) = \Tr(P) - \Tr(Q) \in \Z$.
We conclude that the image of the map \eqref{ImageNeeded} is indeed contained in the subgroup \eqref{TargetSubgroup}, as desired.
\end{proof}

\subsection{Integrality result for Kitaev pairings}\label{SectionIntegralityKitaev}

In this section, we will prove Theorem \ref{MainQuantizationThm} for coarsely transverse partitions.
This will be done by reducing it to a ``standard situation''.

Denote by $\R^{n+1}_{\geq 0} \subset \R^{n+1}$ the set of vectors $x=(x_0, \dots, x_n)$ such that $x_a \geq 0$ for each $a=0,\ldots,n$. 
Define the \emph{corner}
\[
\mathrm{C}^n := \partial \R^{n+1}_{\geq 0}
\]
to be its boundary, whose elements have $x_a = 0$ for at least one $a$. 
Equip $\mathrm{C}^n$ with the metric that takes the maximum of the coordinate distances.
The \emph{standard} coarsely transverse partition of $\mathrm{C}^n$ is defined by
\begin{equation*}
A_i^{\mathrm{std}} = \big\{x = (x_0, \dots, x_n) \in \mathrm{C}^n \mid x_0, \dots, x_{i-1} >0,\,x_i=0\big\},\qquad i=0,\ldots,n.
\end{equation*}

\begin{lemma}\label{LemPartitionClassifyingMap}
Let $A_0,\ldots,A_n$ be a coarsely transverse partition of $M$. There exists a coarse map $f:M\to \mathrm{C}^n$ such that 
\begin{equation}
\label{PullbackStandard}
f^* A^{\mathrm{std}}_i=A_i, \rlap{\qquad $i=0,\ldots,n$.}
\end{equation}
\end{lemma}

\begin{proof}
Define
\begin{equation*}
  f = (d_{A_0}, \dots, d_{A_n}) : M \longrightarrow \mathrm{C}^n,
\end{equation*}
where $d_{A_i}$ denotes the distance-from-$A_i$ function.
One easily checks that \eqref{PullbackStandard} holds.
%
%
Now, the distance function from any subset is Lipschitz continuous with Lipschitz constant one. 
{It follows that $d(f(x), f(y)) \leq d(x, y)$, hence $f$ is controlled.}

It remains to check that $f$ is proper. To this end, let $B \subseteq \mathrm{C}^n$ be bounded. 
Then for some $R>0$, we have
\[
B\subseteq [0, R]^{n+1}\cap \mathrm{C}^n = \bigcap_{i=0}^{n} (A_i^{\mathrm{std}})_R.
\]
Since $f^*A_i^{\mathrm{std}} = A_i$, we get that
\begin{equation*}
  f^{-1}(B) \subseteq \bigcap_{i=0}^{n} (A_i)_R,
\end{equation*}
which is bounded because $A_0, \dots, A_n$ are coarsely transverse. Hence $f$ is proper.
\end{proof}

\begin{lemma}\label{LemHalfSpacesGeneratingPartition}
Let $n\geq 1$, and let $A_0,\ldots,A_n$ be a coarsely transverse partition of $M$. There exists a collection $X_1,\ldots, X_n$ of coarsely transverse half-spaces in $M$ such that
\[
[\mathbf{1}\wedge X_1\wedge\dots\wedge X_n]=(-1)^n n!\cdot [A_0\wedge\dots\wedge A_n] \in HX^n(M).
\]
\end{lemma}

\begin{proof}
Recall the partition of $\R^n$ obtained from the standard half-spaces $X^{\mathrm{std}}_1,\ldots,X^{\mathrm{std}}_n$ in $\R^n$, given in Eq.~\eqref{EuclideanStandardPartition} of Example~\ref{ExampleEuclideanHalfSpaces}. We write it here as
\begin{align*}
\hat{A}^{\mathrm{std}}_0&=X_1^{\mathrm{std}}\cdots X_n^{\mathrm{std}},\\
\hat{A}^{\mathrm{std}}_i&=(X_i^{\mathrm{std}})^c X_{i+1}^{\mathrm{std}}\cdots X_n^{\mathrm{std}},\qquad i=1,\ldots,n.
\end{align*}
It is related to the standard partition of $\mathrm{C}^n$ as follows. 
The orthogonal projection from $\R^{n+1}$ onto the hyperplane
\begin{equation*}
  H^n = \big\{ x = 	(x_0, \dots, x_n) \in \R^{n+1} \mid x_0 + \dots + x_n = 0 \big\}
\end{equation*}
restricts to a bijective coarse equivalence $h:\mathrm{C}^n\to H^n$.
Let $\{e_0,\ldots,e_n\}$ be the standard basis for $\R^{n+1}$, then $\{h(e_1),\ldots,h(e_n)\}$ is a basis for $H^n$, identifying it with $\R^n$. So $h(A_0^{\mathrm{std}}),\ldots,h(A_n^{\mathrm{std}})$ is a coarsely transverse partition of $\R^n$. By iteratively applying further coarse equivalences $\R^n\to \R^n$, we can ``straighten'' this partition to obtain $\hat{A}^{\mathrm{std}}_0,\ldots,\hat{A}^{\mathrm{std}}_n$. (See Fig.~\ref{fig:standard.partition} for the $n=2$ case.) Thus, there is a coarse map $g: \mathrm{C}^n\to\R^n$ such that
\[
g^*\hat{A}^{\mathrm{std}}_i=A^{\mathrm{std}}_i,\qquad i=0,\ldots,n.
\]

Now let $f:M\to \mathrm{C}^n$ be the coarse map of Lemma~\ref{LemPartitionClassifyingMap}, so $f^*A^{\mathrm{std}}_i=A_i$. Then $g\circ f:M\to \R^n$ is a coarse map such that $(g\circ f)^*\hat{A}^{\mathrm{std}}_i=A_i$ for $i=0,\ldots,n$. Now,
\[
X_i=(g\circ f)^*X_i^{\mathrm{std}},\qquad i=1,\ldots,n,
\]
is a  coarsely transverse collection of half-spaces in $M$. Furthermore, by construction,
\begin{align*}
A_0 &=(g\circ f)^*\hat{A}^{\mathrm{std}}_0 =(g\circ f)^*(X_1^{\mathrm{std}}\cdots X_n^{\mathrm{std}})=X_1\cdots X_n,\\
A_i &=(g\circ f)^*\hat{A}^{\mathrm{std}}_i =(g\circ f)^*\big((X_i^{\mathrm{std}})^c X_{i+1}^{\mathrm{std}}\cdots X_n^{\mathrm{std}}\big)=X_i^c X_{i+1}\cdots X_n,\qquad i=1,\ldots,n.
\end{align*}
Now we can apply Corollary \ref{CorEquipartition} to obtain the claim of the lemma.
\end{proof}

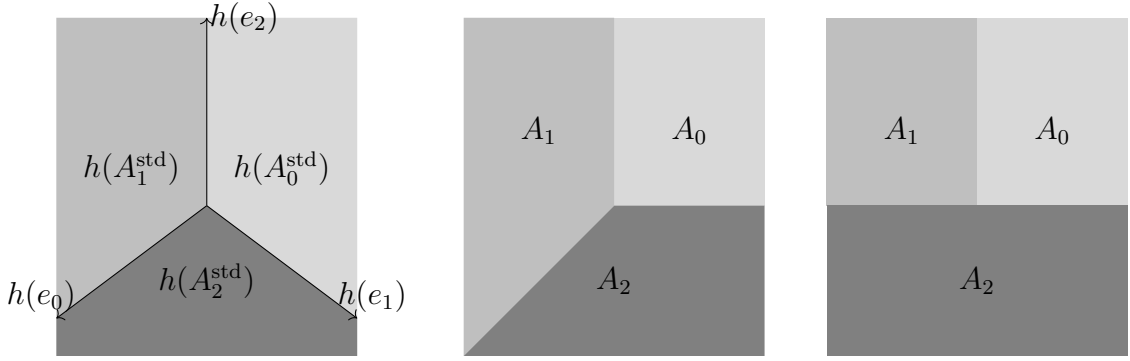
\begin{figure}
\begin{center}
\begin{tikzpicture}
\fill[color=gray!30] (0,0) -- (2,-1.5) -- (2,2.5) -- (0,2.5);
\fill[color=lightgray] (0,0) -- (-2,-1.5) -- (-2,2.5) -- (0,2.5);
\fill[color=gray] (0,0) -- (-2,-1.5) -- (-2,-2) -- (2,-2) -- (2,-1.5);
\draw[->] (0,0) -- (0,2.5);
\draw[->] (0,0) -- (2,-1.5);
\draw[->] (0,0) -- (-2,-1.5);
\node at (1,0.5) {$h(A_0^{\mathrm{std}})$};
\node at (-1,0.5) {$h(A_1^{\mathrm{std}})$};
\node at (0,-1) {$h(A_2^{\mathrm{std}})$};
\node at (-2.5,-1.6) {$h(e_0)$};
\node at (2.5,-1.6) {$h(e_1)$};
\node at (0,2.8) {$h(e_2)$};
\end{tikzpicture}
\hspace{-0.6em}
\begin{tikzpicture}
\fill[color=gray!30] (0,0) -- (2,0) -- (2,2.5) -- (0,2.5);
\fill[color=lightgray] (0,2.5) -- (-2,2.5) -- (-2,-2) -- (0,0);
\fill[color=gray] (2,0) -- (0,0) -- (-2,-2) -- (2,-2);
\node at (1,1) {$A_0$};
\node at (-1,1) {$A_1$};
\node at (0,-1) {$A_2$};
\end{tikzpicture}
~\hspace{1em}
\begin{tikzpicture}
\fill[color=gray!30] (0,0) -- (2,0) -- (2,2.5) -- (0,2.5);
\fill[color=lightgray] (0,2.5) -- (-2,2.5) -- (-2,0) -- (0,0);
\fill[color=gray] (2,0) -- (-2,0) -- (-2,-2) -- (2,-2);
\node at (1,1) {$A_0$};
\node at (-1,1) {$A_1$};
\node at (0,-1) {$A_2$};
\end{tikzpicture}
\end{center}
\caption{
The left figure shows a 2-partition of the hyperplane $H^2$ in $\R^3$ orthogonal to $(1,1,1)$. 
It is projected from the standard 2-partition of the boundary of the positive octant of $\R^3$. 
Let $\{e_0,e_1,e_2\}$ be the standard basis for $\R^3$, and $h$ be the orthogonal projection $\R^3\to H^2$. We may (coarsely) identify $H^2$ with $\R^2$ by mapping the basis $\{h(e_1),h(e_2)\}$ to the standard basis for $\R^2$, yielding a 2-partition of $\R^2$ shown in the middle figure. Applying a suitable coarse self-map $\R^2\to \R^2$, we obtain the 2-partition of $\R^2$ given in Example~\ref{ExampleEuclideanHalfSpaces}.}
\label{fig:standard.partition}
\end{figure}

The quantization of Kitaev pairings to the values stated in Theorem \ref{MainQuantizationThm} now follows immediately from Lemma~\ref{LemHalfSpacesGeneratingPartition} and the quantization of Kubo pairings proved in Section \ref{SectionIntegralityKubo}.

\subsection{A dual index theorem}
\label{SectionDualIndexTheorem}

Let $M$ be a proper metric space and let $\cH$ be an $M$-module.
Throughout, we assume that $\cH$ is ample (see Remark \ref{RemAmpleModule}).
In this case, the $C^*$-algebra obtained by taking the norm closure of $\sB(M) \subset \sL(\cH)$ is denoted by $C^*(M)$ and called the \emph{Roe algebra} of $M$.
Similarly, for a big family $\cY$ in $M$, the norm closure of $\sB(\cY)$ is denoted by $C^*(\cY)$.

For any big family $\cY$ in $M$, we obtain canonical comparison maps
\begin{equation}
\label{ComparisonBRoe}
K^{\mathrm{alg}}_n(\sB(\cY)) \longrightarrow K^{\mathrm{alg}}_n(C^*(\cY)) \longrightarrow KH_n(C^*(\cY)) \longrightarrow K^{\mathrm{top}}_n(C^*(\cY)),
\end{equation}
where the first map is induced by the inclusion $\sB(\cY) \hookrightarrow C^*(\cY)$ and the next maps are the comparison maps \eqref{FactorizationComparisonMap}.

Let $X_1, \dots, X_n$ be a coarsely transverse collection of half-spaces and let $\cX_k$, $0 \leq k \leq n$, be the big families defined in \eqref{DefinitionXk}.
The large diagram from Section~\ref{SectionIntegralityKubo} may then be extended to the commutative diagram
\[
\begin{tikzcd}
K^{\mathrm{alg}}_{n-2m}(\sB(M)^+)
\ar[d]
&
HC_n(\sB(M)^+)
\ar[r, "\widetilde{\chi}_*"]
&
HX_n(M)
\ar[r, "\substack{\text{pairing} \\~\text{with}~\theta_n}"]
&[0.3cm]
\C
\ar[dddddd, equal]
\\
KH_{n-2m}(\sB(M)^+)
\ar[d, "\partial_{\mathrm{MV}}"]
\ar[r]
\ar[ur, "\ch"]
&
K^{\mathrm{top}}_{n-2m}(C^*(M)^+)
\ar[d, "\partial_{\mathrm{MV}}"]
\ar[r, "\beta^m"]
&
K^{\mathrm{top}}_{n}(C^*(M)^+)
\ar[d, "\partial_{\mathrm{MV}}"]
\\
KH_{n-2m-1}(\sB(\cX_{n-1}))
\ar[d, "\partial_{\mathrm{MV}}"]
\ar[r]
&
K^{\mathrm{top}}_{n-2m-1}(C^*(\cX_{n-1})^+)
\ar[d, "\partial_{\mathrm{MV}}"]
\ar[r, "\beta^m"]
&
K^{\mathrm{top}}_{n-1}(C^*(\cX_{n-1})^+)
\ar[d, "\partial_{\mathrm{MV}}"]
\\
\vdots
\ar[d, "\partial_{\mathrm{MV}}"]
&
\vdots
\ar[d, "\partial_{\mathrm{MV}}"]
&
\vdots
\ar[d, "\partial_{\mathrm{MV}}"]
\\
KH_{-2m}(\sB(\cX_0))
\ar[d]
\ar[r]
&
K_{-2m}^{\mathrm{top}}(C^*(\cX_0))
\ar[d, equal]
\ar[r, "\beta^m"]
&
K_{0}^{\mathrm{top}}(C^*(\cX_0))
\ar[d, equal]
\\
KH_{-2m}(\sL_1)
\ar[d]
&
K^{\mathrm{top}}_{-2m}(\sK)
\ar[d, "\beta^m"]
\ar[r, "\beta^m"]
&
K^{\mathrm{top}}_{0}(\sK)
\ar[dl, equal]
\\
K^{\mathrm{top}}_{-2m}(\sL_1)
\ar[r, "\beta^m"]
\ar[ur, "\cong"]
&
K^{\mathrm{top}}_{0}(\sL_1)\cong K^{\mathrm{top}}_{0}(\sK)
\ar[rr, "\Tr/(2\pi i)^m"]
&
&
\C,
\end{tikzcd}
\]
where we used that Bott periodicity commutes with the boundary map in topological $K$-theory.
The vertical composition
\begin{equation}
\label{VerticalComposition}
K_{n}^{\mathrm{top}}(C^*(M)^+) \longrightarrow K_0^{\mathrm{top}}(\sK) \cong \Z
\end{equation}
in the above diagram may be described as a generalized index map, as we now describe. 
To this end, we first observe that the subsets $X_1, \dots, X_n$ determine a certain compactification $\overline{M}$ of $M$, as follows.
First we define bounded continuous functions $\chi_1, \dots, \chi_n$ on $M$ by
\[
\chi_i = \frac{\tilde{\chi}_i}{\sqrt{1 + \tilde{\chi}_1^2 + \dots + \tilde{\chi}_n^2}}, \qquad \text{with} \qquad \tilde{\chi}_i(x) = d(x, X_i^c) - d(x, X_i).
\] 
Let $\sA \subseteq C_b(M)$ be the commutative $C^*$-subalgebra generated by $C_0(M)$ and $\chi_1, \dots, \chi_n$.
By coarse transversality of $X_1, \dots, X_n$, the difference $\chi_1^2 + \dots \chi_n^2 - 1$ is contained in $C_0(M)$, hence $\sA$ is a unital $C^*$-algebra and hence there exists a compact Hausdorff space $\overline{M}$ containing $M$ as dense subspace such that $\sA \cong C(\overline{M})$.
Let $N := \overline{M} \setminus M$ be the boundary of this compactification.
There is a well-defined $*$-homomorphism
\[
 C(S^{n-1}) \longrightarrow C(\overline{M})/C_0(M) \cong C(\overline{M} \setminus M) = C(N)
\]
that sends the canonical coordinate function $x_i$ on $S^{n-1}$ to $\chi_i$.
Via Gel'fand duality, this yields a continuous map 
\begin{equation}
\label{CoronaMap}
\varphi = \varphi_{X_1, \dots, X_n} : N \to S^{n-1}.
\end{equation}
We refer to $N$ as the \emph{spherical corona of} $M$ provided by the collection of half-spaces.

By \cite{QiaoRoe}, there is a bilinear pairing 
\begin{equation}
\label{RoeAlgebraPairing}
  K_i^{\mathrm{top}}(C^*(M)^+) \times K^{i-1}(N) \longrightarrow \Z,
\end{equation}
where $K^{-j}(N) \cong K_j^{\mathrm{top}}(C(N))$ denotes the topological $K$-theory of the space $N$.
First, there is a $*$-homomorphism
\[
C^*(M)^+ \otimes C(N) \longrightarrow C^*(M)^+/\sK, \qquad T\otimes f\longmapsto [T\tilde{f}],
\]
where $\tilde{f} \in C(\overline{M})$ is an arbitrary extension of $f \in C(N)$ to all of $\overline{M}$.
The point here is that if $g \in C_0(M)$, then $Tg$ is compact (hence the map is independent of the choice of representative $\tilde{f}$) and, moreover, for any $\tilde{f} \in C(\overline{M})$ and $T \in C^*(M)$, the commutator $[T, \tilde{f}]$ is compact as well, so the map is a homomorphism (we use here that elements of $C(\overline{M})$ have vanishing variation at infinity, see \cite{QiaoRoe}).
The pairing \eqref{RoeAlgebraPairing} is now obtained as the composition
\begin{align*}
K_i^{\mathrm{top}}(C^*(M)^+) \times K^{i-1}(N) 
&\longrightarrow K_1^{\mathrm{top}}(C^*(M)^+ \otimes C(N)) \\
&\longrightarrow K_1^{\mathrm{top}}(C^*(M)^+/\sK) \\
&\stackrel{\partial}{\longrightarrow} K_0^{\mathrm{top}}(\sK)\cong \Z,
\end{align*}
where the first map is the external product, the second map is induced by the $*$-homomorphism just described and the third map is the boundary map in topological $K$-theory.
By the general description of the external product (e.g.~\cite[Prop.~4.8.3]{HigsonRoe}), the pairing may be described as
\[
\langle [P], [u]\rangle = \ind(P\tilde{u} + 1-P), \qquad \text{or} \qquad \langle [U], [p]\rangle = \ind(U \tilde{p} + 1 - \tilde{p})
\]
in the cases $i=0$ and $i=1$ (mod $2$), respectively.

By joint work of Bunke and the first author \cite[\S4.7]{BunkeLudewig}, the relation of the pairing \eqref{RoeAlgebraPairing} to the iterated Mayer-Vietoris map \eqref{VerticalComposition} is now described as follows.

\begin{theorem}
\label{ThmIndexPairing}
The iterated Mayer-Vietoris map \eqref{VerticalComposition} is equal to the index pairing with the pullback of the canonical generator $\xi_n$ of $K^{n-1}(S^{n-1}) \cong \Z$ along the map \eqref{CoronaMap},
\[
\Tr(\partial_{\mathrm{MV}}^n(x)) = \langle x, \varphi^*\xi_n\rangle, \rlap{\qquad $\forall x \in K_n^{\mathrm{top}}(C^*(M)^+)$.}
\]
\end{theorem}

Rearranging the factor of $(2 \pi i)^m$, one obtains the following result.

\begin{corollary}
\label{CorollaryIndexPairingDiagram}
Let $n \in \N$ and write $n = 2m$ or $n = 2m+1$, depending on whether $n$ is even or odd. 
Then the following diagram commutes:
\begin{equation}
\label{DiagramCorollary}
\begin{tikzcd}
K^{\mathrm{alg}}_{n-2m}(\sB(M)^+)
\ar[r, "\ch"]
\ar[d]
&
HC_n(\sB(M)^+)
\ar[r, "\widetilde{\chi}_*"]
&
HX_n(M)
\ar[dd, "\substack{\text{pairing with} \\~(2 \pi i)^m \theta_n}"]
\\
K^{\mathrm{top}}_{n-2m}(C^*(M)^+)
\ar[d, "\beta^m"']
\\
K^{\mathrm{top}}_{n}(C^*(M)^+)
\ar[r, "\substack{\text{index pairing} \\ \text{with}~\varphi^*\xi_n}"]
&
\Z
\ar[r, hookrightarrow]
&
\C
\end{tikzcd}
\end{equation}
\end{corollary}

We are now able to prove the following theorem.

\begin{theorem}
\label{ThmNontrivialPairing}
For $M = \R^n$ and the standard half-spaces $X_1^{\mathrm{std}}, \dots, X_n^{\mathrm{std}}$ from Example~\ref{ExampleEuclideanHalfSpaces}, the Kitaev and Kubo pairings map onto the subgroup of $\C$ given in Thm.~\ref{MainQuantizationThm}.
\end{theorem}

\begin{proof}
Write $n = 2m$ or $n=2m+1$, depending on whether $n$ is even or odd.
Collecting all constants, observe that we need to show that the image of the composition
\[
K_{n-2m}^{\mathrm{alg}}(\sB(\R^n)^+)
\xrightarrow{~\ch~} HC_n(\sB(\R^n)^+) \xrightarrow{~\tilde{\chi}_*~} HX_n(\R^n) \xrightarrow{~\substack{\text{pairing with} \\ (2 \pi i)^m \theta_n}~} \C
\]
equals $\Z \subset \C$.
By the commutative diagram \eqref{DiagramCorollary}, the image of this map is \emph{contained in} $\Z$, and the map itself is equal to the composition
\begin{equation}
\label{SecondComposition}
K^{\mathrm{alg}}_{n-2m}(\sB(\R^n)^+)
\longrightarrow
 K^{\mathrm{top}}_{n-2m}(C^*(\R^n)^+) \xrightarrow{~\beta^m~}
K^{\mathrm{top}}_n(C^*(\R^n)^+)
\xrightarrow{~\substack{\text{index pairing} \\ \text{with}~\varphi^*\xi_n}~}\Z.
\end{equation}
It remains to show that this map is onto.
Since each of the half-spaces $X_i^{\mathrm{std}}$ is flasque, their Roe algebra $K$-theory groups are trivial  \cite[Prop.~9.3]{Roe-topology} and so each of the Mayer-Vietoris boundary maps
\[
\partial_{\mathrm{MV}} : K_i^{\mathrm{top}}(C^*(\cX_k)) \longrightarrow K_{i-1}^{\mathrm{top}}(C^*(\cX_{k-1}))
\]
is an isomorphism.
In view of Thm.~\ref{ThmIndexPairing}, we conclude that the index pairing with $\varphi^*\xi_n$ is an isomorphism in this case.
Since $C^*(\cX_0) = \sK$, the algebra of compact operators on $\cH$ (whose $K$-theory is given in \eqref{KtheoryCompactOperators}), this implies 
\[
K_i^{\mathrm{top}}(C^*(\R^n)) = 
\begin{cases} 
\Z & i - n\equiv 0 \mod 2
\\
0 & i - n \equiv 1 \mod 2,
\end{cases}
\]
compare also \cite[Thm.~6.4.10]{HigsonRoe}.
We conclude that the composition of the last two maps in \eqref{SecondComposition} is actually an isomorphism (in particular, onto).
It remains to show that the first map in \eqref{SecondComposition} is onto.
This statement hinges on the well-known fact that the non-trivial $K$-theory groups of $C^*(\R^n)$ are generated by the \emph{coarse index class} $\ind(D)$ of the Dirac operator $D$ on $\R^n$ \cite[pp.~33--34]{Roe-topology}. 
%

The coarse index of the Dirac operator may be defined if $M$ is a general complete spin manifold, see for example \cite{Roe-cohom}. 
As shown by Roe in \cite[\S4.3]{Roe-cohom}, if $n=\dim(M)$ is even, then there exists a lift $\ind_{\mathrm{fp}}(D) \in K^{\mathrm{alg}}_0(\sB(M))$ of $\ind(D)$ along the comparison map \eqref{ComparisonBRoe}.
We conclude that the first map in \eqref{SecondComposition} is onto if $n$ is even.

In the case that $n = \dim(M)$ is odd, Roe constructs in \cite[\S4.6]{Roe-cohom} a lift of $\ind(D)$ to $K_{-1}^{\mathrm{alg}}(\sB(M))$, which does not quite fit our setting.
However,  the $C^*$-algebraic coarse index may be defined by the formula
\[
\ind(D) := [w(D)]\in K_1^{\mathrm{top}}(C^*(M)),
\]
for any function $w:\R\to\C^\times$ with $w-1\in C_0(\R)$ and winding number one (of course, any choice of $w$ yields the same $K$-theory class).
Choosing $w$ with compactly supported Fourier transform, the operator $w(D)$ actually defines an invertible element in $\sB(M)^+$ \cite[Prop.~3.6]{Roe-topology}, \cite[Lemma~4.30]{Roe-cohom}, and hence defines a class $\ind_{\mathrm{fp}}(D)$ in $K_1^{\mathrm{alg}}(\sB(M))$ that gets sent to $\ind(D)$ under the comparison map.
This shows that the relevant map is onto $\Z$ in the odd case as well.
%
%
\end{proof}

\begin{remark}\label{RemarkTranslationInvariantVanish}
The class $\ind_{\mathrm{fp}}(D)\in K^{\mathrm{alg}}_0(\sB(M))$ constructed in \cite{Roe-cohom} is represented by idempotents in $M_2(\sB(M)^+)$, where $\sB(M)^+$ denotes the unitization of $\sB(M)$. 
So the non-trivial Kitaev/Kubo pairings there are achieved with the help of the unitization. 

One might ask the question: \emph{If $P=P^2\in M_k(\sB(M))$ is an idempotent over the non-unitized algebra $\sB(M)$, is it possible that $P$ has non-trivial Kubo/Kitaev pairings?} 

For $M=\Z^d$, it is easy to see that this is impossible if $P$ is translation invariant.
We take the $M$-module $\cH = \ell^2(\Z^d)\otimes \C^N$ and let $P$ be an idempotent in the $\Z^d$-invariant subalgebra
\[
\sB(\Z^d)^{\Z^d}\cong \C[\Z^d]\otimes M_N(\C).
\]
Here, $\C[\Z^d]$ denotes the group ring of $\Z^d$, or equivalently, the ring of Laurent polynomials in $d$ variables over $\C$,
\[
\C[\Z^d]=\C[z_1^{\pm1},\ldots,z_d^{\pm1}].
\]
So $P$ represents a class $[P]\in K^{\mathrm{alg}}_0(\sB(\Z^d)^{\Z^d}) = K^{\mathrm{alg}}_0(\C[z_1^{\pm},\ldots,z_d^{\pm}])$. 
Consider the ring homomorphism
\[
\C[\Z^d] \cong \C[z_1^{\pm1}, \dots, z_d^{\pm1}] \longrightarrow \C
\]
given by sending each generator $z_i^{\pm 1}$ to $1 \in \C$.
Observe that the element $P_0 := 1 \otimes e \in \C[\Z^d] \otimes M_N(\C)$, where $e$ is a rank one projection, defines a class in $K_0^{\mathrm{alg}}(\C[\Z^d])$, which is sent to a generator of $K_0^{\mathrm{alg}}(\C) \cong \Z$ under the above map.
On the other hand, it is a standard result that the above map induces an isomorphism on $K_0^{\mathrm{alg}}$ \cite[Corollary 3.2.13]{RosenbergK}, hence $P_0$ is also a generator for $K^{\mathrm{alg}}_0(\C[\Z^d])$. 
We conclude that $[P] = n \cdot [P_0]$ for some $n \in \Z$ and this relation continues to hold in $K^{\mathrm{alg}}_0(\sB(\Z^d))$.
However, $P_0$ has zero propagation, so $P_0A=AP_0A$ for any subset $A\subseteq M$, and consequently, the Kitaev pairings for $P_0$ are trivial.
As Kitaev pairings depend only on $K_0$ classes, it follows that all Kitaev pairings of $P$ must also be trivial.

We mention that similar arguments employing $K^{\mathrm{alg}}_0$ of Laurent polynomial rings were made in \cite{CMST} to show the impossibility of flat and topologically non-trivial energy bands in periodic tight-binding models with strictly finite hopping range.
\end{remark}
\section*{Data availability statement}
Data sharing not applicable to this article as no datasets were generated or analyzed during the current study.

\bibliography{bibinfo}

\end{document}